\documentclass[sn-mathphys]{sn-jnl}

\usepackage{enumitem}  

\theoremstyle{thmstyleone}%
\newtheorem{theorem}{Theorem}[section]
\newtheorem{lemma}{Lemma}[section]

\theoremstyle{thmstyletwo}%
\newtheorem*{remark}{Remark.}%

\theoremstyle{thmstylethree}%

\numberwithin{equation}{section}

\begin{document}

\title[Stabilization of Ginzburg--Landau Spiral Waves]{Pattern-Selective Feedback Stabilization of Ginzburg--Landau Spiral Waves}


\author[1]{\fnm{Isabelle} \sur{Schneider}}\email{isabelle.schneider@fu-berlin.de}

\author[2]{\fnm{Babette} \sur{de Wolff}}\email{b.wolff@vu.nl}

\author*[3]{\fnm{Jia-Yuan} \sur{Dai}}\email{jydai@nchu.edu.tw}

\affil[1]{\orgdiv{Institut f\"{u}r Mathematik}, \orgname{Freie Universit\"{a}t Berlin}, \orgaddress{\street{Arnimallee 7}, \city{Berlin}, \postcode{14195}, \country{Germany}}}

\affil[2]{\orgdiv{Department of Mathematics}, \orgname{Vrije Universiteit Amsterdam}, \orgaddress{\street{De Boelelaan 1111}, \city{HV Amsterdam}, \postcode{1081}, \country{the Netherlands}}}

\affil*[3]{\orgdiv{Department of Applied Mathematics}, \orgname{National Chung Hsing University}, \orgaddress{\street{145 Xingda Rd.}, \city{Taichung City}, \postcode{402}, \country{Taiwan}}}

\abstract{The complex Ginzburg--Landau equation serves as a paradigm of pattern formation and the existence and stability properties of Ginzburg--Landau $m$-armed spiral waves have been investigated extensively. However, many multi-armed spiral waves are unstable and thereby rarely visible in experiments and numerical simulations. In this article we selectively stabilize certain significant classes of unstable spiral waves within circular and spherical geometries. As a result, stable spiral waves with an arbitrary number of arms are obtained for the first time. Our tool for stabilization is the symmetry-breaking control triple method, which is an equivariant generalization of the widely applied Pyragas control to the setting of PDEs.}

\keywords{Ginzburg--Landau equation, $m$-armed spiral waves, symmetry breaking, feedback stabilization, delay control}



\maketitle

\section{Introduction} \label{sec;intro}

We consider the complex Ginzburg--Landau equation 
\begin{equation} \label{complex-GLe}
\partial_t \Psi = (1 + i \, \eta) \, \Delta_{\mathcal{M}} \Psi + \lambda  \left(1- \lvert\Psi\rvert^2 - i \, \beta \, \lvert\Psi\rvert^2\right) \Psi,
\end{equation}
where $\Delta_{\mathcal{M}}$ is the Laplace--Beltrami operator on a compact surface of revolution $\mathcal{M}$ to be defined shortly. Here $\eta \in \mathbb{R}$ is a prescribed complex diffusion parameter, $\lambda > 0$ is a bifurcation parameter, and $\beta \in \mathbb{R}$ is a prescribed kinetic parameter. The unknown function $\Psi$ is complex valued.

Ginzburg--Landau spiral waves are special solutions of \eqref{complex-GLe}, or more precisely relative equilibria, whose shape is recognized by isophase curves emitted from some vortices; see \cite{GoSt03, Mu03}. They play a significant role in studying nonlinear fields in condensed matter physics and hydrodynamic limits. In different contexts vortices are also called phase singularities, topological defects, and wave dislocations; see \cite{Pi99}. Surveys and numerical evidence on Ginzburg--Landau spiral waves are documented in \cite{ArKr02, FiSc03, SaSc20}. 

We aim to understand pattern formation, dynamical behavior, and feedback controls of Ginzburg--Landau spiral waves on the surface $\mathcal{M}$. To this end we present a trilogy of research: existence, stability analysis, and feedback stabilization. The first two episodes regarding existence and stability analysis have been investigated extensively in \cite{DaRong20, Se05} and also by Dai in \cite{Da20, DaLa21}. This article serves as the third episode in which we stabilize certain classes of unstable spiral waves by introducing noninvasive symmetry-breaking feedback controls with spatio-temporal delays. For this purpose we adopt the \emph{control triple method} introduced by Schneider in \cite{Sch16, Sch18}. 

Existence of spiral waves can be triggered by symmetry-breaking bifurcations (see \cite{ChLa00, Va82}); a fact we will exploit for construction of the control terms.
It has been proved in \cite{Ha82, KoHo81} that spiral waves of \eqref{complex-GLe} exist on the plane $\mathbb{R}^2$. Since in experiments and numerical simulations the underlying domain is bounded, in \cite{Da20, DaLa21} Dai carried out a global bifurcation analysis and proved the existence of spiral waves in circular and spherical geometries. 

For stability analysis, the shooting method used in \cite{DaLa21} allows us to estimate the unstable dimension of spiral waves for sufficiently small parameters $0 \le \lvert\eta\rvert, \, \lvert\beta\rvert \ll 1$ in \eqref{complex-GLe}. Since only stable spiral waves are observable in experiments or numerical simulations, we are interested in whether the unstable spiral waves obtained in the literature \cite{DaRong20, Da20, DaLa21, Se05, Ts10} become locally exponentially stable by introducing suitable feedback controls. 

The control term used in this article is inspired by the Pyragas control scheme introduced in \cite{PYR92}, one of the most successful methods to control the local stability of equilibria or periodic orbits of the ODE system $\dot{z}(t) = f(z(t))$ with $z(t) \in \mathbb{R}^n$. The spirit of Pyragas control is to keep the targeted solution unchanged, while its local stability property is steered as desired. Concretely, the control scheme reads
\begin{equation} 
\dot{z}(t)=f(z(t))+b \, \left(z(t)-z(t-\tau)\right) \quad \mbox{for   } z(t) \in \mathbb{R}^n,
\end{equation}
where the matrix $b \in \mathbb{R}^{n \times n}$ is called the \emph{feedback gain}. The \textit{control term} $b \, (z(t) - z(t-\tau))$ is often called \emph{noninvasive} since it vanishes on equilibria and on periodic solutions with period $\tau > 0$. Pyragas control is widely applied in experimental and numerical settings, because it renders the unstable targeted solutions visible while its implementation is model-independent and requires no expensive calculations; see  \cite{LEK95,OMA12,SCH06, SCH93,YAM09}. Mathematical results on  Pyragas control, however, are delicate and rely on explicit properties of the model; see \cite{FIE20, HOE05, Sch16a, proefschrift, oddnumber}.  In the setting of PDEs, feedback controls of Pyragas type have been exploited for solutions which are periodic in space or time; see \cite{LU96, MON04,POS07}.

The \emph{control triple method} adapts the spirit of Pyragas control to the setting of equivariant PDEs with the aim to stabilize spatio-temporal patterns. To this end, we consider the following control system for the Ginzburg--Landau equation \eqref{complex-GLe}:
\begin{equation} \label{complex-GLe_control}
\partial_t \Psi = (1 + i \, \eta) \, \Delta_{\mathcal{M}} \Psi + \lambda  \left(1-\lvert\Psi\rvert^2 - i \, \beta \, \lvert\Psi\rvert^2\right) \Psi + b \, \left(\Psi - \mathcal{C}_{(h, \tau, g)}[\Psi] \right),
\end{equation}
where $b \in \mathbb{R}$ is the feedback gain and $\mathcal{C}_{(h, \tau, g)}[\Psi]$ denotes the \textit{control operator} given by
\begin{equation} \label{control-operator}
\mathcal{C}_{(h, 
\tau, g)}[\Psi](t,x) := h \, \Psi(t-\tau, gx) \quad \mbox{for    } t \ge 0, \, x \in \mathcal{M}.
\end{equation}
The control operator transforms the \emph{output signal} $\Psi$ by a \emph{multiplicative factor} $h \in \mathbb{C}$, a \emph{time delay} $\tau \geq 0$, and a \emph{space shift} $g: \mathcal{M} \to \mathcal{M}$ induced by the equivariance of \eqref{complex-GLe}. The three ingredients
\[
\left(\mbox{multiplicative factor   } h \in \mathbb{C}, \mbox{   time delay   } \tau \geq 0, \mbox{   space shift   } g: \mathcal{M} \to \mathcal{M}\right).
\]
characterize the control operator \eqref{control-operator} and are also referred to as the \emph{control triple}. 
Given a targeted solution $\Psi_\ast$ of \eqref{complex-GLe}, we choose the control triple in such a way that  $\Psi_\ast = \mathcal{C}_{(h, \tau, g)} [\Psi_\ast]$. Therefore, the control term $b \, (\Psi - C_{(h, \tau, g)}[\Psi])$ vanishes on the targeted solution $\Psi_*$, and thereby the targeted solution is also a solution of the control system \eqref{complex-GLe_control}.

The control term $b \, (\Psi - C_{(h, \tau, g)}[\Psi])$ is \emph{selective} in the sense that it only preserves targeted solutions with the prescribed spatio-temporal symmetries. Therefore, it allows us to select and stabilize certain unstable spatio-temporal solutions  (e.g., spiral waves) over all competing patterns. Moreover, the control term $b \, (\Psi - C_{(h, \tau, g)}[\Psi])$ is \emph{symmetry-breaking} in the sense that it  uses a proper subset of the set of spatio-temporal symmetries of the targeted solution. The terminology `symmetry-breaking control term' is inspired by -- and in line with -- the terminology `symmetry-breaking bifurcation'. 

This article is organized as follows: In Section \ref{sec:simple-example}, we explain the core ideas and design of symmetry-breaking control of spiral waves with an illustrative example. In Section \ref{sec;setting}, we review the general mathematical setting for studying Ginzburg--Landau spiral waves in circular and spherical geometries and then provide the relevant existence and (in-)stability results from the literature. In Section \ref{sec;results}, we derive the precise formulation of the control triple and state our main results, namely that we can stabilize selected spiral waves. Finally, Section \ref{sec:feedback-stabilization} is devoted to the proof of our main results.

\section{From equivariance to control - an example} \label{sec:simple-example}

In order to convey the main ideas of this article, in this section we discuss the feedback control of spiral waves in a key example. In the light of accessibility, this section is set up with as little generality as possible; the rigorous mathematical setting and more general statement will be discussed in Section \ref{sec;setting} and Section \ref{sec;results}.

Concretely, we consider the Ginzburg-Landau equation \eqref{complex-GLe} on the unit $2$-sphere $\mathcal{M} = S^2$ and with parameter values $(\eta, \beta) = (0, 0)$, so that \eqref{complex-GLe} becomes
\begin{equation} \label{complex-GLe-simple}
\partial_t \Psi =  \Delta_{S^2} \Psi + \lambda  \left(1- \lvert\Psi\rvert^2 \right) \Psi.
\end{equation}
We parametrize $S^2$ by spherical coordinates
\begin{equation} \label{spherical-coordinates-simple}
S^2 = \left\{ ( \sin(s) \cos(\varphi), \sin(s) \sin(\varphi), \cos(s)) : s \in  [0, \pi], \, \varphi \in S^1 \cong \mathbb{R}/2 \pi \mathbb{Z} \right\}.
\end{equation}
The PDE \eqref{complex-GLe-simple} possesses a \emph{global gauge symmetry} in the sense that 
\begin{equation} \label{gauge-symmetry}
\Psi(t, s, \varphi) \mbox{   is a solution of (\ref{complex-GLe})} \mbox{   if and only if   } e^{i \omega} \Psi(t, s, \varphi) \mbox{   is a solution }
\end{equation}
for each $\omega \in S^1$. 
Moreover, \eqref{complex-GLe-simple} has a rotational symmetry on the $\varphi$-variable, i.e., 
\begin{equation} \label{rotation-symmetry-simple}
\Psi(t,s,\varphi) \mbox{   is a solution of (\ref{complex-GLe})} \mbox{   if and only if   } \Psi(t,s,\varphi-\zeta)  \mbox{   is a solution},
\end{equation}
for each $\zeta \in S^1$; and \eqref{complex-GLe-simple} has a reflection symmetry, i.e. 
\begin{equation} \label{reflection-symmetry-simple}
\Psi(t,s,\varphi) \mbox{   is a solution of (\ref{complex-GLe})} \mbox{   if and only if   } \Psi(t,\pi-s,\varphi) \mbox{   is a solution}.
\end{equation}
The equivariance relations \eqref{gauge-symmetry}--\eqref{reflection-symmetry-simple} motive us to seek \emph{$m$-armed spiral wave solutions} satisfying the Ansatz 
\begin{equation} \label{spiral-ansatz}
\Psi(t, s, \varphi) := e^{-i \Omega t} \, u(s) \, e^{i m \varphi},
\end{equation}
where $m \in \mathbb{N}$ is the number of arms, $\Omega \in \mathbb{R}$ is the rotation frequency, and the radial part $u(s) \in \mathbb{C}$ is either \textit{even-symmetric}, i.e., $u(\pi -s) = u(s)$, or \textit{odd-symmetric}, i.e., $u(\pi-s) = -u(s)$. 

For each fixed number of arms $m \in \mathbb{N}$, $m$-armed spiral waves of the form \eqref{spiral-ansatz} exist as was proven by Dai in \cite{Da20}; they bifurcation from the trivial solution $\Psi \equiv 0$ at an infinite sequence of bifurcation values
\begin{equation} \label{eigenvalues-order}
0 < \lambda^m_0 < \lambda^m_1 < ... < \lambda^m_k < ..., \quad \lim_{k \rightarrow \infty} \lambda^m_k = \infty;
\end{equation}
see Figure \ref{fig;diagram} in Section \ref{sec;setting}. Moreover, every $m$-armed spiral wave that bifurcates at the bifurcation value $\lambda^m_k$ for $k \in \mathbb{N}_0$ has the following \textit{$\mathbb{Z}_2$-radial-symmetry}:
\begin{equation} \label{radial-symmetry-simple}
u(\pi -s ) = (-1)^k u(s) \quad \mbox{for   } s \in [0, \pi].
\end{equation}
For the specific parameters $(\eta, \beta) = (0, 0)$, it holds that rotation frequency $\Omega = 0$ due to the gradient dynamics induced by a strict Lyapunov functional; see \eqref{lyapunov-functional}. Hence every spiral wave solution of \eqref{complex-GLe-simple} is in fact an equilibrium and we also call it a \emph{vortex equilibrium}.

It has been proven that all $m$-armed spiral waves on the sphere $S^2$ are not locally exponentially stable; see \cite[Theorem 1.2]{DaRong20} and \cite[Theorem 1.3]{DaLa21}. Hence they serve as ideal candidates to be stabilized. 
To this end, we select an $m$-armed spiral wave that bifurcates from the bifurcation value $\lambda^m_j$ for some $j \in \mathbb{N}_0$ and denote this spiral wave by $\Psi_j$. Since $\Psi_j$ is an equilibrium and additionally satisfies \eqref{spiral-ansatz} and \eqref{radial-symmetry-simple}, it holds that
\begin{equation}
\Psi_j(t, s, \varphi) = (-1)^j e^{i m \zeta} \Psi_j(t-\tau, \pi- s, \varphi-\zeta), 
\end{equation} 
for every $\tau \ge 0$ and $\zeta \in S^1$. Consequently, a control term of the form 
\begin{equation} \label{control-term-simple}
b \left( \Psi - (-1)^j e^{i m \zeta} \Psi(t-\tau, \pi-s, \varphi-\zeta) \right), 
\end{equation}
with $\Psi = \Psi(t, s, \varphi)$ and $b \in \mathbb{R}$, vanishes on the selected spiral wave $\Psi_j$. As a result, $\Psi_j$ is also an equilibrium of the control system
\begin{equation} \label{control-system-simple}
\begin{split}
\partial_t \Psi & =  \Delta_{S^2} \Psi + \lambda  \left(1-\lvert\Psi\rvert^2 \right) \Psi 
\\&
\quad + b \, \left(\Psi - (-1)^j \, e^{i m \zeta} \, \Psi(t -\tau, \pi-s, \varphi - \zeta) \right),
\end{split}
\end{equation}
Our task is now to find $\tau \ge 0$, $\zeta \in S^1$, and $b \in \mathbb{R}$ such that the selected spiral wave $\Psi_j$ becomes a locally exponentially stable solution of the control system \eqref{control-system-simple}. Here the choice of parameter $\zeta \in S^1$ determines in which way the control term is \emph{pattern-selective}, i.e. it determines which spiral waves (other than the selected wave $\Psi_j$) are preserved by \eqref{control-system-simple}. Note that the space shift $\varphi - \zeta$ also pins the spiral tips to both poles of sphere. 

In the proof of the stabilization results, the main idea is that the control term \eqref{control-term-simple} should \textit{not} vanish on the unstable and center eigenfunctions associated with the selected spiral wave. Our stability analysis in Section \ref{sec:feedback-stabilization}, which is based on the Fourier decomposition \eqref{fourier-decomposition}, shows that the eigenfunctions associated with the selected spiral wave are of the form $v(s) \, e^{i n \varphi}$ with $n \in \mathbb{Z}$. Since there are only finitely many unstable and center eigenfunctions, all but finitely many choices of $\zeta \in S^1$ ensure that the control term \eqref{control-term-simple} does not vanish on all unstable and center eigenfunctions. 

We emphasize that the control term \eqref{control-system-simple} exploits all the known symmetries of spiral waves in the literature; see \cite{Da20, DaLa21}. In particular, the $\mathbb{Z}_2$-radial-symmetry \eqref{radial-symmetry-simple} of the radial part allows us to stabilize all $m$-armed spiral waves with $j = 0, 1$.

\begin{theorem}[Selective stabilization of $m$-armed spiral waves on the sphere] \label{theorem-simple}  Fix $m \in \mathbb{N}$, $\lambda > \lambda_j^m$ with $j \in \{0,1\}$, and let 
\begin{equation} \label{spiral-wave-simple}
\Psi_j(t, s,  \varphi ) =   u_j(s) \, e^{im\varphi}
\end{equation}
be the $m$-armed spiral wave of the Ginzburg--Landau equation \eqref{complex-GLe-simple}. Then for all but finitely many choices of $\zeta \in S^1$, there exists a constant $\tilde{b}(\zeta) < 0$ such that each feedback gain $b \le \tilde{b}$ admits an upper bound $\tilde{\tau} = \tilde{\tau}(\zeta, b) > 0$ for which $\Psi_j$ becomes a locally exponentially stable solution of the control system \eqref{control-system-simple} for all time delays $\tau \in [0, \tilde{\tau})$.
\end{theorem}

In the next section, we introduce the general setting in which we study stabilization of Ginzburg--Landau spiral waves. There we consider the Ginzburg--Landau equation on more general surfaces on revolution; such surfaces maintain the rotation symmetry and also include disks that are topologically different from spheres. Moreover, we include parameters $(\eta, \beta) \neq (0,0)$, for which most spiral waves are rotating. 

\section{Setting, existence, and (in)stability}\label{sec;setting}

Throughout this article we consider a compact surface of revolution $\mathcal{M}$, which we parametrize by polar coordinates
\begin{equation} \label{polar-coordinates}
\mathcal{M} := \left\{ ( a(s) \cos(\varphi), a(s) \sin(\varphi), \tilde{a}(s)) : s \in  [0, s_*], \, \varphi \in S^1 \cong \mathbb{R}/2 \pi \mathbb{Z} \right\}.    
\end{equation}
Two main examples of $\mathcal{M}$ are the unit disk (when $a(s) = s$ and $\tilde{a}(s) = 0$ for $s \in [0,1]$) and the unit 2-sphere (when $a(s) = \sin(s)$ and $\tilde{a}(s) = \cos(s)$ for $s \in [0,\pi]$). In general, we make the following assumptions on the surface $\mathcal{M}$ and its parametrization: 
\begin{enumerate}
\item The function $a$ satisfies 
\begin{equation} \label{property-a(s)}
a(0) = 0 \mbox{ and } a(s)>0 \mbox{ for } s \in (0, s_*).
\end{equation}
\item The smoothness class of $\mathcal{M}$ is $C^{2, \upsilon}$ with a fixed H\"older exponent $\upsilon \in (0, 1)$. Equivalently, $a$ and $\tilde{a}$ are $C^{2, \upsilon}$ functions. Moreover, $\tilde{a}'(0) = 0$ because the smoothness of $\mathcal{M}$ prevents formation of a cusp at $s = 0$. 
\item We let $s$ be the arc length parameter, i.e., $(a'(s))^2 + (\tilde{a}'(s))^2 = 1$ for $s \in [0,s_*]$;
\end{enumerate}
Topologically, we distinguish the surface $\mathcal{M}$ between two cases: We say that $\mathcal{M}$ has \emph{circular geometry} if its boundary $\partial \mathcal{M}$ is nonempty; otherwise we say that  $\mathcal{M}$ has \emph{spherical geometry}. In the latter case, we restrict ourselves to the situation where $\mathcal{M}$ has reflection symmetry, i.e., we additionally assume the following: 
\begin{enumerate}
\item[4.] If $\partial \mathcal{M}$ is empty, we assume 
\begin{equation} \label{reflection-symmetry}
a(s) = a(s_* -s) \quad \text{ for $s \in [0, s_*]$}.
\end{equation}
\end{enumerate}
Note that $\partial \mathcal{M}$ is empty if and only if $a(s_*) = 0$, due to \eqref{polar-coordinates}--\eqref{property-a(s)}.

For both circular and spherical geometries, we consider $\Delta_{\mathcal{M}}: \mathrm{Dom}(\Delta_\mathcal{M})  \rightarrow L^2(\mathcal{M}, \mathbb{C})$, where the domain $\mathrm{Dom}(\Delta_\mathcal{M})$ is chosen as $H^2(\mathcal{M}, \mathbb{C})$, and if $\partial \mathcal{M}$ is nonempty, it is also equipped with the following Robin boundary conditions:
\begin{equation}\label{robin-boundary}
\alpha_1  \Psi+\alpha_2  \nabla \Psi \cdot \textbf{n}= 0.
\end{equation}
Here $\textbf{n}$ is the unit outer normal vector field on $\partial \mathcal{M}$; the scalars $\alpha_1, \alpha_2\in\mathbb{R}$ are not both zero and $\alpha_1 \alpha_2 \geq 0$. The latter assumption is technical and is required for the global bifurcation analysis in \cite{Da20}. Robin boundary conditions \eqref{robin-boundary} generalize Neumann (or so-called \textit{no-flux}) boundary conditions (i.e., $\alpha_1 = 0$) frequently adopted in applied settings and for numerical studies;  see \cite{ArBi98, JIA04, Se05, Ts10}. In addition, for $(\eta, \beta) = (0,0)$ in the Ginzburg--Landau equation \eqref{complex-GLe}, the more general Robin boundary conditions \eqref{robin-boundary} have been derived by minimizing a free energy in the theory of superconductivity; see \cite{Duetal92}. 

Following \cite{GoSt03}, we graphically exhibit an $m$-armed spiral wave \eqref{spiral-ansatz} by plotting the level set where the imaginary part of \eqref{spiral-ansatz} is zero. Hence as we express $u(s) = A(s) \, e^{i p(s)}$ in the polar form, we obtain the level set where the \emph{phase field} $- \Omega t + p(s) + m \varphi$ of $\Psi(t,s, \varphi)$ is equal to zero modulo $\pi$. The $2\pi$-periodicity of angle $\varphi$ then yields the relations 
\begin{equation} 
\varphi =  \varphi_\ell(t, s) =  \frac{\Omega t - p(s)  + \ell \pi}{m} \,\,\, (\mathrm{mod} \mbox{   } 2\pi) \quad \mbox{for    } \ell = 0,1,...,2m-1,
\end{equation}
and we plot the pattern on $\mathcal{M}$ via the coordinates \eqref{polar-coordinates}. In this way, the pattern associated with \eqref{spiral-ansatz} is exhibited as a twisted spiral, motivating the name \emph{spiral wave}. We interpret \emph{vortices} of a spiral wave as phase singularities, i.e., zeros of $\Psi$ at which the phase field of $\Psi$ undergoes a jump discontinuity. Indeed, the Fourier mode $e^{im\varphi}$ of $\Psi$ on the $\varphi$-variable implies that the vortices reside at $s = 0$, and also at $s = s_*$ if $\partial \mathcal{M}$ is empty; see Fig. \ref{FIGspiral}.

\begin{figure}[ht]%
\centering
\includegraphics[width=1\textwidth]{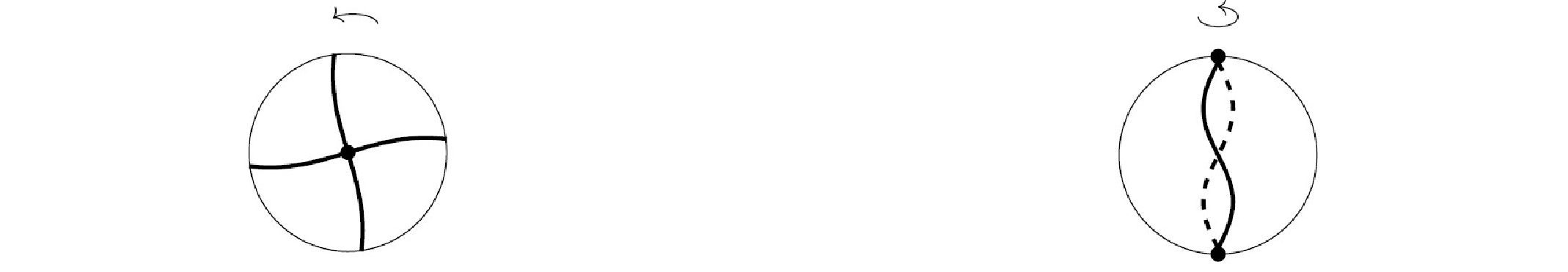}
\caption{On the left, a $2$-armed spiral pattern on the disk with the origin as the vortex. On the right, a $1$-armed spiral pattern on the sphere with the north and south poles as the vortices. Both spiral patterns may rotate with respect to the axis of rotation of the surface $\mathcal{M}$ with the rotation frequency $\Omega$. This figure, including its caption, has previously been published in \cite{DaLa21}.}\label{FIGspiral}
\end{figure}

Due to the gauge symmetry \eqref{gauge-symmetry} the $L^2$-subspace
\begin{equation}
L_m^2(\mathbb{C}) := \left\{ \psi \in L^2(\mathcal{M}, \mathbb{C}) : \psi(s, \varphi) = u(s) \, e^{im\varphi}, \, u(s) \in \mathbb{C} \right\}
\end{equation}
is invariant under the dynamics of the Ginzburg--Landau equation \eqref{complex-GLe}. Note that the restriction 
\begin{equation}
\Delta_{m} := \Delta_{\mathcal{M}} \big\lvert_{L_m^2(\mathbb{C})}: L_m^2(\mathbb{C})\rightarrow L_m^2(\mathbb{C})
\end{equation}
is well defined. Indeed, in polar coordinates \eqref{polar-coordinates} we read $\Delta_m$ as follows:
\begin{equation} \label{restricted-operator}
\Delta_m \left(u(s) \, e^{im \varphi}\right) = \left( u''(s) + \frac{a'(s)}{a(s)} u'(s) - \frac{m^2}{a^2(s)} u(s) \right) e^{i m \varphi}.
\end{equation}
Substituting the Ansatz \eqref{spiral-ansatz} with $\psi(s, \varphi) := u(s) \, e^{im\varphi}$ into \eqref{complex-GLe} yields the following elliptic PDE on $L_m^2(\mathbb{C})$, for which we call the \textit{spiral wave equation}:
\begin{equation} \label{spiral-wave-equation}
0 = (1 + i \, \eta) \, \Delta_{m}\psi + i \, \Omega \, \psi + \lambda \left(1 - \lvert\psi\rvert^2 - i \, \beta \, \lvert\psi\rvert^2 \right)  \psi.
\end{equation}

Dai proved in \cite{Da20} that nontrivial solutions of \eqref{spiral-wave-equation}, in the sense that $\psi$ is not identically zero, form countably many supercritical pitchfork bifurcation curves as the parameter $\lambda$ crosses the simple eigenvalues $\lambda^m_k$ of $-\Delta_{m}$. Moreover, we order the set of all bifurcation values $\{\lambda_k^m : k \in \mathbb{N}_0\}$ as in \eqref{eigenvalues-order}. We quote the following existence result of spiral waves by Dai from \cite[Lemma 2.4 and Lemma 3.5 (iii)]{Da20} and \cite[Theorem 1.2]{DaLa21}. It is worthy emphasizing that the obtained spiral waves possess not only an arbitrary number of arms $m \in \mathbb{N}$, but also an arbitrary \textit{nodal class} of the radial part $j \in \mathbb{N}_0$.

\begin{lemma}[Existence] \label{lemma-existence}
For each fixed $m \in \mathbb{N}$, $k \in \mathbb{N}_0$, and $\lambda \in (\lambda^m_k, \lambda^m_{k+1}]$ there exists an $\varepsilon > 0$ such that the spiral wave equation \eqref{spiral-wave-equation} possesses $k+1$ distinct \emph{(}up to a gauge symmetry \eqref{gauge-symmetry}\emph{)} nontrivial solution-pairs parametrized by $\eta,\beta\in (-\varepsilon,\varepsilon)$ and $\eta \neq \beta$, denoted by
\begin{equation}
(\Omega(\eta, \beta), \psi_j(\cdot, \cdot \,\lvert\, \eta, \beta)) \in \mathbb{R} \times L_m^2(\mathbb{C}), \quad j = 0,1,...,k,
\end{equation}
and the following statements hold: 
\begin{enumerate}[label=(\roman*)]
\item[\emph{(i)}] \emph{(}$\mathbb{Z}_2$-radial-symmetry\emph{)} Suppose that in addition $\partial \mathcal{M}$ is empty and the reflection symmetry \eqref{reflection-symmetry} holds. Then 
\begin{equation}
\psi_j(s_*-s, \varphi \,\lvert\,\eta, \beta) = (-1)^j \psi_j(s, \varphi \,\lvert\, \eta, \beta)
\end{equation}
for $s \in [0, s_*]$ and $\varphi \in S^1$.
\item[\emph{(ii)}] $\Omega(0,0) = 0$ and the radial part $u_{j}(s)$ of $\psi_j(s,\varphi) := \psi_j(s, \varphi \,\lvert\, 0,0)$ is real valued and possesses $j$ simple zeros on $(0, s_*)$.
\end{enumerate}
Moreover, we classify the types of patterns as shown in Fig. \ref{fig;diagram}.
\end{lemma}

\begin{figure}[h]%
\centering
\includegraphics[width=1\textwidth]{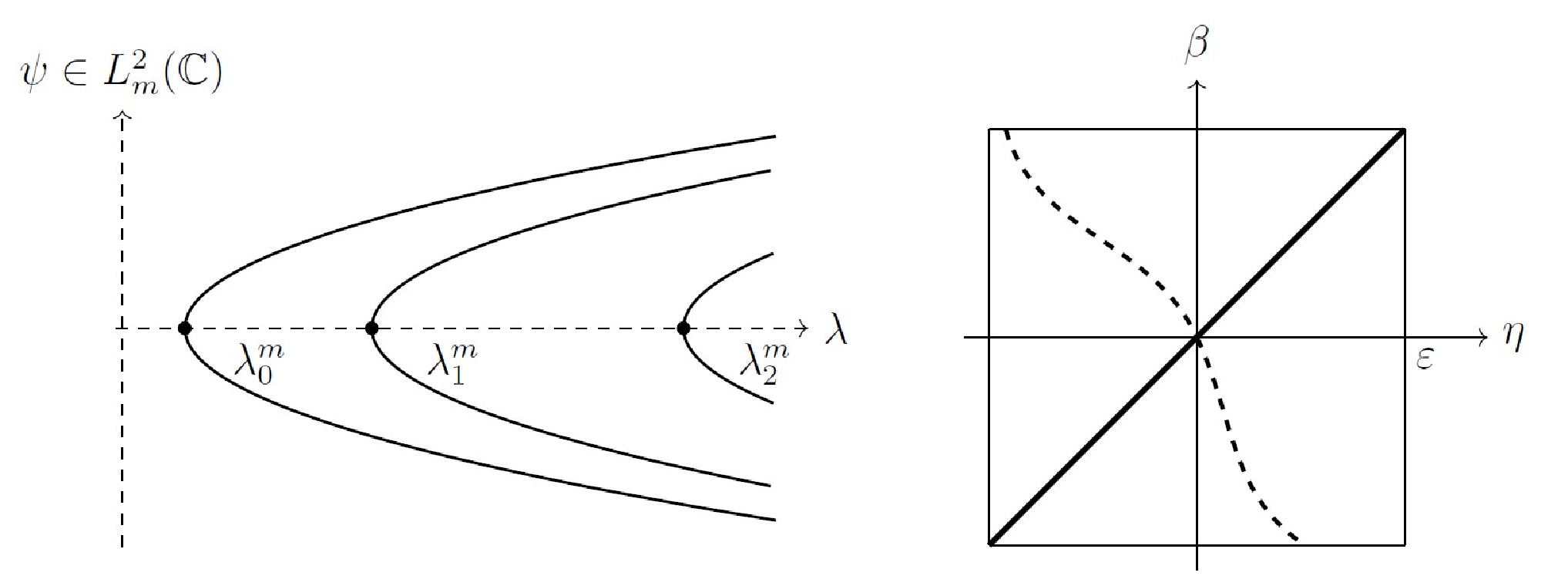}
\caption{On the left, the global bifurcation diagram of the spiral wave equation (\ref{spiral-wave-equation}). The shape of each bifurcation curve is a supercritical pitchfork. On the right: different types of pattern in the $(\eta,\beta)$-parameter space, according to \cite[Lemma 5.5]{Da20}. Each parameter not on the bold line yields a spiral pattern as shown in Fig. \ref{FIGspiral}. Such spiral patterns are rotating, i.e., $\Omega \neq 0$, if and only if parameters do not lie on the dashed line. The right panel of this figure, including its caption, has previously been published in \cite{DaLa21}.}\label{fig;diagram}
\end{figure}

We next collect stability information for the spiral waves obtained in Lemma \ref{lemma-existence}. We now fix $m \in \mathbb{N}$ and $\lambda > 0$. Every spiral wave is a nontrivial equilibrium of
\begin{equation} \label{evolution-pde}
\partial_t U = \mathcal{F}(\Omega, U \,\lvert\, \eta, \beta) := (1 + i \, \eta) \, \Delta_{\mathcal{M}} U + i \, \Omega \, U + \lambda \left(1 - \lvert U \rvert^2 - i \, \beta\, \lvert U \rvert^2 \right) U,
\end{equation}
where $\mathcal{F} : \mathbb{R} \times \mathrm{Dom}(\Delta_\mathcal{M}) \times \mathbb{R} \times \mathbb{R} \rightarrow L^2(\mathcal{M}, \mathbb{C})$ is well defined due to the continuous embedding of $\mathrm{Dom}(\Delta_\mathcal{M})$ into $L^2(\mathcal{M}, \mathbb{C})$. Here we recall that $\mathrm{Dom}(\Delta_\mathcal{M})$ is $H^2(\mathcal{M}, \mathbb{C})$, and is also equipped with Robin boundary conditions \eqref{robin-boundary} if $\partial \mathcal{M}$ is nonempty. The PDE \eqref{evolution-pde} generates a local semiflow on the interpolation space $H^{2 \alpha}(\mathcal{M}, \mathbb{C})$ for any fixed exponent $\alpha > 1/2$, according to \cite[Theorem 3.3.3]{He81}. Moreover, the local stability of a spiral wave $(\Omega(\eta, \beta), \psi_j(\cdot, \cdot \,\lvert\, \eta, \beta))$ obtained in Lemma \ref{lemma-existence} is determined by the spectrum of the partial Fr\'{e}chet derivative 
\begin{equation} \label{linearization-general}
\mathcal{L}_j(\eta, \beta):= \partial_U \mathcal{F}(\Omega(\eta, \beta), \psi_j(\cdot, \cdot \,\lvert\, \eta, \beta) \,\lvert\, \eta, \beta) : L^2(\mathcal{M}, \mathbb{C}) \rightarrow L^2(\mathcal{M}, \mathbb{C})
\end{equation}
with the domain $\mathrm{Dom}(\mathcal{L}_j(\eta, \beta)) := \mathrm{Dom}(\Delta_{\mathcal{M}})$; see \cite[Chapter 5]{He81}. Notice that $\mathcal{L}_{j}(\eta, \beta)$ is an $\mathbb{R}$-linear operator, because we always identify $\mathbb{C}$ with $\mathbb{R}^2$ as a real vector space.

Since $\mathcal{L}_j(\eta, \beta)$ is a uniformly elliptic operator on a bounded surface $\mathcal{M}$, it has compact resolvent and thus its spectrum, denoted by $\sigma(\mathcal{L}_j(\eta, \beta))$, consists of eigenvalues with finite algebraic multiplicity; see \cite{Ka95}. The gauge symmetry \eqref{gauge-symmetry} always triggers a \textit{trivial eigenvalue}, which is zero; its associated eigenfunctions span the tangent space along the group orbit of the spiral wave. Since $\mathcal{L}_j(\eta, \beta)$ is a sectorial operator, the following quantity is well defined:
\begin{equation} \label{principal-eigenvalue-general}
\begin{aligned}
\mu_j^*(\eta, \beta) :=
\begin{cases}
\max \left\{\mathrm{Re}(z) : z \in \sigma(\mathcal{L}_j(\eta, \beta)) \setminus \{0\} \right\} & \mbox{if $0$ is algebraically simple}; \\
\max \left\{\mathrm{Re}(z) : z \in \sigma(\mathcal{L}_j(\eta, \beta)) \right\} &\mbox{if otherwise.}
\end{cases}
\end{aligned}
\end{equation}

It follows that the spiral wave $(\Omega(\eta, \beta), \psi_j(\cdot, \cdot \,\lvert\, \eta, \beta))$ is locally exponentially stable (resp., unstable) if $\mu_j^*(\eta, \beta) < 0$ (resp., $\mu_j^*(\eta, \beta) >0$). 

\begin{lemma} \label{lemma-upper-semicontinuous}
The upper bound $\mu_j^*(\eta, \beta)$ depends upper-semicontinuously on the parameters $(\eta, \beta) \in \mathbb{R}^2$. Consequently, the vortex equilibrium $\psi_j(\cdot, \cdot \,\lvert\, 0,0)$ is locally exponentially stable \emph{(}resp., unstable\emph{)} if and only if the spiral wave $(\Omega(\eta, \beta), \psi_j(\cdot, \cdot \,\lvert\, \eta, \beta))$ is locally exponentially stable \emph{(}resp., unstable\emph{)} for sufficiently small parameters $0 \le \lvert\eta\rvert, \, \lvert\beta\rvert \ll 1$.
\end{lemma}
\begin{proof}[Proof.]
Let $\{\mathcal{S}_j(t \,\lvert\, \eta, \beta)\}_{t \ge 0}$ be the linear semiflow on $L^2(\mathcal{M}, \mathbb{C})$ generated by $\mathcal{L}_j(\eta, \beta)$. It suffices to show that the spectrum of $\mathcal{S}_j(t \,\lvert\, \eta, \beta)$ depends upper-semicontinuously on the parameters $(\eta, \beta)$ for each fixed $t > 0$. Since the spectrum of $(1 + i \, \eta) \, \Delta_{\mathcal{M}}$ is the same as the spectrum of $\Delta_\mathcal{M}$ multiplying by $1+ i\,\eta$ and the reaction term of $\mathcal{L}_j(\eta, \beta) - \mathcal{L}_j(0,0)$ is a bounded $L^2$-perturbation, $\mathcal{S}_j(t \,\lvert\, \eta, \beta)$ converges to $\mathcal{S}_j(t \,\lvert\, 0,0)$ in the operator norm for each fixed $t > 0$; see \cite[Chapter 9, Theorem 2.16]{Ka95}. Hence the spectrum of $\mathcal{S}_j(t \,\lvert\, \eta, \beta)$ depends upper-semicontinuously on the parameters $(\eta, \beta)$ for each fixed $t > 0$; see \cite[Chapter 4, Remark 3.3]{Ka95}.
\end{proof}

Since spiral waves in Lemma \ref{lemma-existence} are known to exist only for sufficiently small parameters $0 \le \lvert\eta \rvert,\, \lvert\beta\rvert \ll 1$, by the upper-semicontinuous dependence in Lemma \ref{lemma-upper-semicontinuous} we now focus on the \textit{variational case} $(\eta, \beta) = (0,0)$. Then the Ginzburg--Landau equation \eqref{complex-GLe} is associated with the following \textit{strict Lyapunov functional} (also see \cite{ArKr02}):
\begin{equation} \label{lyapunov-functional}
\mathcal{E}[\Psi] := \int_\mathcal{M}  \lvert\nabla \Psi \rvert^2 - \lambda \left( \lvert\Psi\rvert^2 - \frac{\lvert\Psi\rvert^4}{2} \right) \, \mathrm{d}V_\mathcal{M} + \frac{\alpha_1}{\alpha_2} \int_{\partial \mathcal{M}} \lvert\Psi\rvert^2 \, \mathrm{d}V_{\partial \mathcal{M}}.
\end{equation}
Here $\mathrm{d}V_{\mathcal{M}}$ and $\mathrm{d}V_{\partial \mathcal{M}}$ stand for the volume elements on $\mathcal{M}$ and $\partial \mathcal{M}$, respectively. Note that the boundary integral is absent if $\partial \mathcal{M}$ is empty, or in case of either Neumann ($\alpha_1 = 0$) or Dirichlet ($\alpha_2 = 0$) boundary conditions; see \eqref{robin-boundary}.

Recall the notation $\psi_j(s, \varphi) := \psi_j(s, \varphi \,\lvert\, 0,0)$ and let $\mu_j^* := \mu_j^*(0,0)$. Since $\Omega(0,0) = 0$ by Lemma \ref{lemma-existence} (ii), the elliptic equation \eqref{spiral-wave-equation} for $(\eta,\beta) = (0,0)$ reads
\begin{equation} \label{vortex-equation}
0 = \Delta_{m}  \psi_j + \lambda \left( 1 - \lvert\psi_j\rvert^2 \right)  \psi_j,
\end{equation}
and thus we also say that $\psi_j$ is a \textit{vortex equilibrium} of the Ginzburg--Landau equation \eqref{complex-GLe}. Since the radial part $u_j(s)$ of $\psi_j(s, \varphi)$ is real valued, it holds that
\begin{equation}
\mathcal{L}_j[V] := \mathcal{L}_j(0,0)[V] = \Delta_{\mathcal{M}} V + \lambda \left(  \left(1 - 2 \, \lvert\psi_j\rvert^2 \right) V - \lvert\psi_j\rvert^2 \, e^{2 i m \varphi} \, \overline{V} \right),
\end{equation}
where $\overline{V}$ denotes the complex conjugate of $V$. It follows that $\mathcal{L}_j$ is self-adjoint on $L^2(\mathcal{M}, \mathbb{C})$ with respect to the following inner product:
\begin{equation}
\langle V_1, V_2 \rangle_{L^2} := \mathrm{Re} \, \left( \int_\mathcal{M} V_1 \, \overline{V_2} \, \mathrm{d}V_{\mathcal{M}} \right).   
\end{equation}
We define the \emph{principal eigenvalue} of $\mathcal{L}_j$ as the largest nontrivial eigenvalue, which indeed coincides with $\mu_j^*$ as defined in \eqref{principal-eigenvalue-general}.

We first collect well-known instability results for the nodal class $j = 0$. In this class, by Lemma \ref{lemma-existence} (ii) the radial part $u_0(s)$ of $\psi_0(s, \varphi)$ does not change sign on $(0, s_*)$. Indeed, the vortex equilibrium $\psi_0$ is a minimizer of the Lyapunov functional \eqref{lyapunov-functional}. 

\begin{lemma}[Instability for $j =0$] \label{lem-instability-0} 
The following statements hold:
\begin{itemize}
\item[\emph{(i)}] Let $\mathcal{M}$ be the unit disk equipped with Neumann boundary conditions. Then $\mu_0^* > 0$ for any fixed $m \in \mathbb{N}$ and
sufficiently large $\lambda > 0$; see \emph{\cite[Theorem 1.3]{Se05}}.
\item[\emph{(ii)}] Let $\mathcal{M}$ be the unit $2$-sphere. Then $\mu_0^* \ge 0$ for any fixed $m \in \mathbb{N}$ and $\lambda > \lambda_0^m$; see \emph{\cite[Theorem 1.2]{DaRong20}}.
\end{itemize}
\end{lemma}

\begin{remark}
In the case of circular geometry, we consider homogeneous boundary conditions \eqref{robin-boundary}, while we are well aware of the existence and stability results with inhomogeneous Dirichlet boundary conditions $\Psi(t, s_*, \varphi) = e^{i \tilde{m} \varphi}$ for fixed $\tilde{m} \in \mathbb{N}$; see \cite{Mi05} for instance. In fact, our choice of boundary conditions originates from the application to feedback control. An important feature of our results is that the control is \emph{pattern-selective}, i.e., we are able to select and stabilize certain spiral waves over all other spiral waves present in the uncontrolled system. Homogeneous boundary conditions highlight this feature, since \emph{all} $m$-armed spiral waves are present in the Ginzburg--Landau equation \eqref{complex-GLe}. In contrast, for the inhomogeneous Dirichlet boundary conditions $\Psi(t, s_*, \varphi) = e^{i \tilde{m} \varphi}$ only the $\tilde{m}$-armed spiral waves are present in \eqref{complex-GLe}, and hence the inhomogeneity already restricts (or, in a sense, `selects') the spiral waves.
\end{remark}

The next lemma asserts the instability of vortex equilibria for all other nodal classes $j \in \mathbb{N}$, for which the radial part $u_j(s)$ of $\psi_j(s,\varphi)$ changes sign exactly $j$-times on $(0, s_*)$. In this case, the instability is caused by radial perturbations and the proof is based on a shooting argument; see \cite[Theorem 1.3]{DaLa21}.

\begin{lemma}[Instability for $j \in \mathbb{N}$] \label{lem-instability-1} For both circular and spherical geometries, if $j \in \mathbb{N}$, then $\mu_j^* > 0$ for any fixed $m \in \mathbb{N}$ and $\lambda > \lambda^m_j$.
\end{lemma}

The instability results in Lemma \ref{lem-instability-0} and Lemma \ref{lem-instability-1} motivate us to stabilize those unstable $m$-armed spiral waves for sufficiently small parameters $0 \le \lvert\eta\rvert, \, \lvert\beta\rvert \ll 1$ through noninvasive symmetry-breaking controls.

\section{Symmetry-breaking controls and main results} \label{sec;results}

We recall that the surface of revolution $\mathcal{M}$ has rotational symmetry on the $\varphi$-variable, and reflection symmetry when the boundary $\partial \mathcal{M}$ is empty. These symmetries allow us to design the control triple explicitly. More precisely, we define the control operator as
\begin{equation} \label{control-operator-explicit}
\mathcal{C}_{(h, 
\tau, (\iota, \zeta))}[\Psi](t,s, \varphi) := h \, \Psi(t-\tau, R_\iota(s), \varphi - \zeta), 
\end{equation}
where $h \in \mathbb{C}$ is a multiplicative factor, $\tau \geq 0$ is a time delay, and the space shift, denoted by $(\iota, \zeta) \in \{+,-\} \times S^1$, consists of 
\begin{align} 
R_{\iota}(s) := \left\{
\begin{array}{ll}
s \quad & \mbox{if   } \iota = +,
\\
s_* - s & \mbox{if   } \iota = -, \mbox{   when   } \partial \mathcal{M} \mbox{   is empty},
\end{array} \right.
\end{align}
and a rotation $\zeta \in S^1$ on the $\varphi$-variable. With this notation, we consider the following control system for the Ginzburg--Landau equation \eqref{complex-GLe}:
\begin{equation} \label{control-system}
\begin{split}
\partial_t \Psi & = (1 + i \, \eta) \, \Delta_{\mathcal{M}} \Psi + \lambda  \left(1-\lvert\Psi\rvert^2 - i \, \beta \, \lvert\Psi\rvert^2 \right) \Psi 
\\&
\quad + b \, \left(\Psi - h \, \Psi(t-\tau, R_{\iota}(s), \varphi - \zeta) \right),
\end{split}
\end{equation}
where $\Psi = \Psi(t, s, \varphi)$. We call 
\begin{equation} \label{control-term-explicit}
b \, (\Psi - h \, \Psi(t-\tau, R_{\iota}(s), \varphi - \zeta))
\end{equation}
a symmetry-breaking control because its design is based on the spatio-temporal symmetries of the targeted spiral waves and not on the full equivariance of the uncontrolled system. In our Ginzburg--Landau setting, spiral waves obtained in Lemma \ref{lemma-existence} are triggered by symmetry-breaking bifurcation from the trivial equilibrium $\Psi \equiv 0$ under the following $(S^1 \times \Gamma_\mathcal{M})$-equivariance:
\begin{equation} \label{equivariance}
\left((\omega, \gamma)\Psi\right)(t, x) := e^{-i \omega} \, \Psi \left(t, \gamma^{-1}x\right) \quad \mbox{for   } t \ge 0, \,x \in \mathcal{M}.
\end{equation}
Here $S^1$ results from the gauge symmetry \eqref{gauge-symmetry} and $\Gamma_\mathcal{M}$ is a matrix group containing symmetries of $\mathcal{M}$. In polar coordinates \eqref{polar-coordinates}, when $\partial \mathcal{M}$ is nonempty, then $\Gamma_{\mathcal{M}} = \textbf{SO}(2, \mathbb{R}) \cong S^1$ and $\gamma^{-1}$ in \eqref{equivariance} induces a rotation $-\zeta \in S^1$ of the $\varphi$-variable in the control term \eqref{control-term-explicit}. When $\partial \mathcal{M}$ is empty, then $\Gamma_\mathcal{M} = \textbf{O}(2,\mathbb{R})$ due to the reflection symmetry \eqref{reflection-symmetry} on $\mathcal{M}$, and the reflection $x \mapsto -x$ induces $R_-(s) = s_* - s$ in the control term \eqref{control-term-explicit}. 

For fixed $(\eta, \beta) \in \mathbb{R}^2$, $m \in \mathbb{N}$, $\lambda > \lambda_j^m$, and $j \in \mathbb{N}_0$, let
\begin{equation} \label{general-spiral-wave}
\Psi_j(t, s, \varphi \,\lvert\, \eta, \beta) := e^{-i \Omega(\eta, \beta) t} \, u_j(s \,\lvert\, \eta, \beta) \, e^{im\varphi}
\end{equation}
be a solution satisfying the Ansatz \eqref{spiral-ansatz} obtained in Lemma \ref{lemma-existence}. 
We now determine multiplicative factors $h \in \mathbb{C}$ so that the control term \eqref{control-term-explicit} is noninvasive on $\Psi_j$, i.e.,
\begin{equation} \label{noninvasive-equation}
\Psi_j(t,s, \varphi \,\lvert\, \eta, \beta) - h \, \Psi_j(t-\tau, R_\iota(s), \varphi-\zeta  \,\lvert\, \eta, \beta) = 0 
\end{equation}
holds for $t \ge 0$, $s \in [0, s_*]$, and $\varphi \in S^1$. There are two cases.

Case 1: $\iota = +$ and thus $R_+(s) = s$. Substituting \eqref{general-spiral-wave} into \eqref{noninvasive-equation} yields
\begin{equation} \label{noninvasive-h-0}
h = h(\tau, \zeta \,\lvert\, \eta, \beta) = e^{i(-\Omega(\eta,\beta) \tau + m \zeta)}.
\end{equation}
This case does not require any symmetry assumptions of $\mathcal{M}$ on the $s$-variable, and so it is applicable for both circular and spherical geometries.

Case 2: $\iota = -$ and thus $R_-(s) = s_*-s$. Substituting \eqref{general-spiral-wave} into \eqref{noninvasive-equation} implies that
\begin{equation} \label{noninvasive-equation-h-new}
u_j(s \,\lvert\, \eta, \beta) - h \, e^{i \Omega(\eta, \beta) \tau} u_j(s_* - s \,\lvert\, \eta, \beta) \, e^{-im \zeta} = 0 
\end{equation}
holds for $s \in [0, s_*]$. The $\mathbb{Z}_2$-radial-symmetry in Lemma \ref{lemma-existence} (i) allows us to choose 
\begin{equation} \label{noninvasive-h-1}
h = h(\tau, \zeta \,\lvert\, \eta, \beta) = (-1)^j \, e^{i(-\Omega(\eta, \beta) \tau + m \zeta)}.
\end{equation}
As a result, we choose the multiplicative factors $h \in \mathbb{C}$ in \eqref{control-term-explicit} as follows:
\begin{equation} \label{noninvasive-general}
h = h(\tau, \zeta \,\lvert\, \eta, \beta)
= \left\{
\begin{array}{ll}
e^{i(-\Omega(\eta, \beta) \tau + m \zeta)}, & \mbox{if   } \iota = +,
\\ 
(-1)^j \, e^{i(-\Omega(\eta, \beta) \tau + m \zeta)}, & \mbox{if   } \iota = -, \, \mbox{when   } \partial \mathcal{M} \mbox{   is empty}.
\end{array} \right.
\end{equation}

From \eqref{noninvasive-general} we see that at the selected spiral wave $\Psi_j$ the time delay $\tau \ge 0$ itself induces an \textit{external rotation} $\Psi_j \mapsto e^{-i \Omega(\eta, \beta) \tau}\, \Psi_j$, while the space shift $\zeta \in S^1$ itself induces another external rotation $\Psi_j \mapsto e^{im\zeta}\, \Psi_j$. Hence on $\Psi_j$ time delays and space shifts are interchangeable. However, for the whole control system \eqref{control-system} time delays and space shifts trigger very different dynamical effects. Time delays and space shifts are also different from the viewpoint of implementation: For purely spatial control (i.e., $\tau = 0$) the rotation frequency $\Omega(\eta, \beta) \in \mathbb{R}$ does not appear in the control term, and thus purely spatial control can also be implemented when $\Omega(\eta, \beta)$ is unknown.

Intuitively, the spirit of feedback stabilization is that while the control term \eqref{control-term-explicit} vanishes on
the selected spiral wave $\Psi_j$, it should not vanish on the space spanned by all unstable and center eigenfunctions associated with $\Psi_j$.
Since spiral waves are only known to exist for sufficiently small parameters $0 \le \lvert\eta\rvert, \, \lvert\beta\rvert \ll 1$ (see Lemma \ref{lemma-existence}), it suffices the consider the variational case $(\eta, \beta) = (0,0)$ for their local stability analysis (see Lemma \ref{lemma-upper-semicontinuous}). For the case $(\eta, \beta) = (0,0)$ it holds that $\Omega(0,0) = 0$ and thus $\Psi_j = \psi_j$ is a vortex equilibrium and \eqref{noninvasive-general} becomes
\begin{equation} \label{noninvasive-variational}
h = h(\tau, \zeta \,\lvert\, 0, 0) =
\left\{
\begin{array}{ll}
e^{im \zeta}, & \mbox{if   } \iota = +,
\\ 
(-1)^j \, e^{i m \zeta}, & \mbox{if   } \iota = -, \, \mbox{when   } \partial \mathcal{M} \mbox{   is empty}.
\end{array} \right.
\end{equation}
When $\iota = +$, the control term \eqref{control-term-explicit} vanishes on all eigenfunctions $v(s) \, e^{im\varphi}$ in $L_m^2(\mathbb{C})$ and stabilization is only possible if $\psi_j$ is already locally exponentially stable in $L_m^2(\mathbb{C})$, that is, only if $j = 0$; see the spectral structure on $L_m^2(\mathbb{C})$ in Lemma \ref{lemma-spectral-no-control-Lm} (i). On the other hand, when $\iota = -$, the control term vanishes on eigenfunctions $v(s) \, e^{im\varphi}$ in $L_m^2(\mathbb{C})$ that are either even-symmetric (i.e., $v(s_* - s) = v(s)$) or odd-symmetric (i.e., $v(s_*-s) = -v(s)$). In this case, stabilization is possible only if $j = 0, 1$; see Lemma \ref{lemma-spectral-no-control-Lm}. So the control term \eqref{control-term-explicit} only allows us to aim stabilization for the two nodal classes: $j = 0$, and also $j = 1$ when $\partial \mathcal{M}$ is empty. 

Our main results consist of two theorems, which assert that stabilization is indeed achieved for the two nodal classes: The first theorem applies to the class $j = 0$ in both circular and spherical geometries, and the second theorem applies to the class $j = 1$ in spherical geometry.

\begin{theorem}[Selective stabilization of $m$-armed spiral waves for $j = 0$ in circular and spherical geometries] \label{theorem-0} Fix $m \in \mathbb{N}$, $\lambda > \lambda_0^m$, and let 
\begin{equation} \label{spiral-wave-0}
\Psi_0(t, s,  \varphi \,\lvert\, \eta, \beta) = e^{-i \Omega(\eta, \beta) t} \, \psi_0(s, \varphi \,\lvert\, \eta, \beta)
\end{equation}
be the $m$-armed spiral wave obtained in Lemma \ref{lemma-existence}, where $\psi_0(s, \varphi \,\lvert\, \eta, \beta) = u_0(s \,\lvert\, \eta, \beta) \, e^{im\varphi}$ has the even-symmetric radial part $u_0$.

Then for all but finitely many choices of $\zeta \in S^1$, there exists a constant $\tilde{b} = \tilde{b}(\zeta) < 0$ such that each $b \le \tilde{b}$ admits a constant $\tilde{\tau} = \tilde{\tau}(\zeta, b) > 0$ for which $\Psi_0$ becomes a locally exponentially stable solution of the control system
\begin{equation} \label{control-system-0}
\begin{split}
\partial_t \Psi & = (1 + i \, \eta) \, \Delta_{\mathcal{M}} \Psi + \lambda  \left(1-\lvert\Psi\rvert^2 - i \, \beta \, \lvert\Psi\rvert^2 \right) \Psi 
\\&
\quad + b \, \left(\Psi - e^{i(-\Omega(\eta, \beta) \tau + m \zeta)} \, \Psi(t-\tau, s, \varphi - \zeta) \right)
\end{split}
\end{equation}
for all $\tau \in [0, \tilde{\tau})$ and $\eta, \beta \in (-\varepsilon, \varepsilon)$ with $\varepsilon > 0$ small enough.
\end{theorem}

\begin{theorem}[Selective stabilization of $m$-armed spiral waves for $j = 1$ in spherical geometry] \label{theorem-1} Suppose that $\partial \mathcal{M}$ is empty and the reflection symmetry \eqref{reflection-symmetry} holds. Fix $m \in \mathbb{N}$, $\lambda > \lambda_1^m$, and let 
\begin{equation} \label{spiral-wave-1}
\Psi_1(t, s,  \varphi \,\lvert\, \eta, \beta) = e^{-i \Omega(\eta, \beta) t} \, \psi_1(s, \varphi \,\lvert\, \eta, \beta)
\end{equation}
be the $m$-armed spiral wave obtained in Lemma \ref{lemma-existence}, where $\psi_1(s, \varphi \,\lvert\, \eta, \beta) = u_1(s \,\lvert\, \eta, \beta) \, e^{im\varphi}$ has the odd-symmetric radial part $u_1$.

Then for all but finitely many choices of $\zeta \in S^1$, there exists a constant $\tilde{b} = \tilde{b}(\zeta) < 0$ such that each $b \le \tilde{b}$ admits a constant $\tilde{\tau} = \tilde{\tau}(\zeta, b) > 0$ for which $\Psi_1$ becomes a locally exponentially stable solution of the control system
\begin{equation} \label{control-system-1}
\begin{split}
\partial_t \Psi & = (1 + i \, \eta) \, \Delta_{\mathcal{M}} \Psi + \lambda  \left(1-\lvert\Psi\rvert^2 - i \, \beta \, \lvert\Psi\rvert^2 \right) \Psi 
\\&
\quad + b \, \left(\Psi - (-1) \, e^{i(-\Omega(\eta, \beta) \tau + m \zeta)} \, \Psi(t-\tau, s_*-s, \varphi - \zeta) \right)
\end{split}
\end{equation}
for all $\tau \in [0, \tilde{\tau})$ and $\eta, \beta \in (-\varepsilon, \varepsilon)$ with $\varepsilon > 0$ small enough.
\end{theorem}

We provide three remarks regarding Theorem \ref{theorem-0} and Theorem \ref{theorem-1}. 
\begin{itemize}
\item First, the finitely many exceptions of $\zeta \in S^1$ for stabilization are determined by the unstable dimension of the spiral waves; see the proof of Lemma \ref{lemma-pure-spatial-stabilization-0} and Lemma \ref{lemma-pure-spatial-stabilization-1}. Lower bound estimates of the unstable dimension have been investigated (see \cite{DaRong20, DaLa21} for instance), but in general the exact value of the unstable dimension remains unknown.
\item  Second, pure temporal controls (i.e., $\iota = +$ and $\zeta  = 0$ in \eqref{control-system}) cannot achieve stabilization, as we will prove in Lemma \ref{lemma-failure}. Hence space shifts play an indispensable role for stabilization. Failure of stabilization with pure time delays has also been documented for different models; see \cite{Sch16, Sch18, oddnumber}.
\item As a direct consequence, for the variational case $(\eta, \beta) = (0,0)$ we can selectively stabilize all the unstable vortex equilibria obtained in Lemma \ref{lem-instability-0} and also those with the nodal class $j = 1$ in Lemma \ref{lem-instability-1}, independently of the number of arms $m \in \mathbb{N}$.
\end{itemize}

Our stabilization results, Theorem \ref{theorem-0} and Theorem \ref{theorem-1}, are novel in the following four aspects.
\begin{itemize}
\item[(1)] For the first time, $m$-armed spiral wave solutions of the complex Ginzburg--Landau equation \eqref{complex-GLe} are successfully stabilized. Moreover, stabilization is achieved for an arbitrary number of arms $m \in \mathbb{N}$.
\item[(2)] We stabilize spiral waves \textit{selectively}, in the sense that only the targeted spiral wave with the prescribed spatio-temporal symmetries is stabilized. In particular, depending on the symmetry of the underlying surface $\mathcal{M}$, we distinguish spiral waves by their nodal class $j \in \{0,1\}$ of the radial part, and then stabilize them.
\item[(3)] We can stabilize spiral waves along the \textit{global} bifurcation curves, in the sense that stabilization is achieved for an arbitrary bifurcation parameter $\lambda$ strictly larger than the relevant bifurcation values $\lambda_0^m, \lambda_1^m$. Consequently, spiral waves far away from the trivial equilibrium, which thus possess large amplitudes, can be stabilized. 
\item[(4)] We stabilize Ginzburg--Landau solutions with an \emph{inhomogeneous} amplitude. In contrast, the relevant literature on feedback stabilization in the Ginzburg--Landau equation considered explicit solutions with homogeneous amplitude; see \cite{MON04,POS07}. 
\end{itemize}

We indicate two directions of future research based on our stabilization results. First, regarding mathematical analysis, we can investigate stabilization of spiral waves within the nodal classes $j \ge 1$ in circular geometry and $j \ge 2$ in spherical geometry, respectively. The control term \eqref{control-term-explicit} already exhausts all the known symmetries of $m$-armed spiral waves. So to obtain further stabilization results, the main task is to first obtain more spatio-temporal symmetries of spiral waves than the $\mathbb{Z}_2$-radial-symmetry (see Lemma \ref{lemma-existence} (i)), and then to design more general symmetry-breaking control terms.  

Second, regarding scientific applications, we expect that numerical implementation and experimental realization of our stabilization results can be carried out. Spiral waves in various models have been investigated extensively in experiments, and spatially extended feedback methods are realized for example through illumination for the photosensitive Belousov--Zhabotinsky reaction \cite{KHE02, VAN00}, with the help of an electrocardiogram in cardiac tissue \cite{PAN00}, or by regulating the carbon monoxide partial pressure in catalytic carbon oxidation of platinum \cite{KIM01}.

Last, we emphasize that the design of our feedback control terms relies on symmetry arguments alone. Hence it is by no means limited to the specific setting of the Ginzburg--Landau equation, and we expect our control method to be widely applicable theoretically, numerically, and experimentally.

\section{Proof of selective feedback stabilization} \label{sec:feedback-stabilization}

In this section we prove the main results Theorem \ref{theorem-0} and Theorem \ref{theorem-1}.
Our proof consists of four steps. For the first three steps we consider the variational case $(\eta, \beta) = (0,0)$ where $\Psi_j = \psi_j$ is a vortex equilibrium. First, we study the spectral structure of the linearization operator at $\psi_j$ without control. Second, we achieve stabilization by pure space shifts (i.e., $\tau = 0$ in \eqref{control-system}). 
Third, we show that such stabilization persists under sufficiently small time delays $0 < \tau \ll 1$. In the final fourth step, we complete the proof of Theorem \ref{theorem-0} and Theorem \ref{theorem-1} by ensuring that stabilization persists under sufficiently small parameters $0 \le \lvert\eta\rvert,\, \lvert\beta\rvert \ll 1$.

\subsection{Spectral structure without control}\label{subsec:no-control}

For the variational Ginzburg--Landau equation without control (i.e., $(\eta, \beta) = (0,0)$ and $b = 0$ in \eqref{control-system}), the local stability of an $m$-armed vortex equilibrium $\psi_j$ is determined by solutions of the following linear evolutionary equation (see \cite[Chapter 5]{He81}):
\begin{align} \label{linear-equation-no-control}
\partial_t V = \mathcal{L}_j[V] := \Delta_\mathcal{M} V + \lambda  \left( \left(1- 2 \,u_j^2 \right) V - u_j^2 \, e^{2 i m \varphi} \, \overline{V} \right).
\end{align}
To simplify the analysis, we apply the change of coordinates
\begin{equation} \label{shift-no-control}
W(t,s,\varphi) := V(t,s,\varphi) \, e^{-im\varphi}
\end{equation}
which shifts the index of the Fourier modes on the $\varphi$-variable. Then in polar coordinates \eqref{polar-coordinates} we see that \eqref{linear-equation-no-control} is equivalent to 
\begin{align} \label{linear-equation-no-control-W}
\partial_t W = \, & \Delta_{\mathcal{M}} W + \frac{2im}{a^2} \partial_\varphi W - \frac{m^2}{a^2} W + \lambda  \left( \left(1-2 u_j^2\right) W  - u_j^2 \, \overline{W} \right).
\end{align}

We sort out the real and imaginary parts of $W$ by setting $W = P + i Q$, where $P, Q$ are real-valued functions. Then \eqref{linear-equation-no-control-W} is equivalent to
\begin{align} \label{linear-equation-no-control-P}
\partial_t P = \, & \Delta_{\mathcal{M}} P - \frac{2m}{a^2} \partial_\varphi Q - \frac{m^2}{a^2} P + \lambda  \left(1-3 u_j^2\right)  P,
\\ \label{linear-equation-no-control-Q}
\partial_t Q = \, &\Delta_{\mathcal{M}} Q + \frac{2m}{a^2} \partial_\varphi P - \frac{m^2}{a^2} Q + \lambda \left(1- u_j^2\right)  Q.
\end{align} 

Since $\mathcal{L}_j$ defined in \eqref{linear-equation-no-control} is self-adjoint and has compact resolvent, the following Fourier decomposition holds: 
\begin{equation} \label{fourier-decomposition}
L^2(\mathcal{M}, \mathbb{C}) = \bigoplus_{n \in \mathbb{Z}} L_n^2(\mathbb{C}), 
\end{equation} 
where $L_n^2(\mathbb{C}) := \left\{ \psi \in L^2(\mathcal{M}, \mathbb{C}) : \psi(s, \varphi) = u(s) \, e^{in\varphi}, \, u(s) \in \mathbb{C} \right\}$. This allows us to substitute the following exponential Ansatz on the $t$-variable and Fourier Ansatz on the $\varphi$-variable into the system \eqref{linear-equation-no-control-P}--\eqref{linear-equation-no-control-Q}: 
\begin{align} 
P(t,s,\varphi) = e^{t \mu}\sum_{n \in \mathbb{Z}} P_n(s,\varphi), \quad 
Q(t,s,\varphi) =  e^{t\mu}\sum_{n \in \mathbb{Z}} Q_n(s,\varphi).
\end{align}
Here $\mu \in \mathbb{R}$ is an eigenvalue of $\mathcal{L}_j$ and $(P_n, Q_n) \in (L_n^2(\mathbb{C}))^2 := L_n^2(\mathbb{C}) \times L_n^2(\mathbb{C})$. Then the system \eqref{linear-equation-no-control-P}--\eqref{linear-equation-no-control-Q} is equivalent to countably many eigenvalue problems for $\mathcal{L}_j$ restricted to $(L_n^2(\mathbb{C}))^2$ and indexed by $n \in \mathbb{Z}$: 
\begin{align} \label{linear-equation-no-control-Pn}
\mu \, P_n = \, & \Delta_{n} P_n - \frac{2imn}{a^2} Q_n - \frac{m^2}{a^2} P_n + \lambda  \left(1-3 u_j^2\right)  P_n,
\\ \label{linear-equation-no-control-Qn}
\mu \, Q_n  = \, & \Delta_{n} Q_n + \frac{2im n}{a^2} P_n - \frac{m^2}{a^2} Q_n + \lambda  \left(1- u_j^2\right)  Q_n.
\end{align} 
Here $\Delta_n$ is the restriction of $\Delta_\mathcal{M}$ to $L_n^2(\mathbb{C})$. We denote by
\begin{equation} 
\mathcal{L}_{j, n} := \mathcal{L}_j \big\lvert_{(L_n^2(\mathbb{C}))^2}: (L_n^2(\mathbb{C}))^2 \rightarrow (L_n^2(\mathbb{C}))^2.
\end{equation}

The following spectral properties of $\mathcal{L}_{j,n}$ are inherited from $\mathcal{L}_j$.

\begin{lemma} \label{lemma-spectral-no-control} The spectrum $\sigma(\mathcal{L}_{j,n})$ consists of real eigenvalues, only. The principal eigenvalue $\mu_{j,n}^*$ of $\mathcal{L}_{j,n}$ exists and satisfies
\begin{equation} \label{largest-constant}
\mu_{j,n}^* \le \mu_j^* \quad \mbox{for   } n \in \mathbb{Z},
\end{equation}
where $\mu_j^*$ is the principal eigenvalue of $\mathcal{L}_j$; see \eqref{principal-eigenvalue-general}. Moreover, for each fixed $j \in \mathbb{N}_0$ we have
\begin{equation} \label{finite-unstable-modes}
\lim_{\lvert n\rvert \rightarrow \infty}\mu_{j,n}^* = -\infty.
\end{equation}
Consequently, there is an $n_j \in \mathbb{N}_0$ such that $\mu_{j,n}^* < 0$ if $\lvert n \rvert \ge n_j$. 
\end{lemma}
\begin{proof}[Proof.]
Since $\mathcal{L}_{j,n}$ is the restriction of the uniformly elliptic operator $\mathcal{L}_j$, it is sectorial and has compact resolvent. Hence  $\sigma(\mathcal{L}_{j,n})$ consists of eigenvalues with finite multiplicity and the intersection between $\sigma(\mathcal{L}_{j,n})$ and any vertical strip in $\mathbb{C}$ is a finite set. Therefore, since for fixed $j \in \mathbb{N}_0$ the set $\{\mu_{j,n}^*:  n \in \mathbb{Z} \}$ is infinite, $\lim_{\lvert n\rvert \rightarrow \infty}\mu_{j,n}^* = -\infty$ holds. 
\end{proof}

The stability analysis in Subsection \ref{subsec;stabilization} requires more spectral information about $\mathcal{L}_{j, 0}$. By definition, $\mu \in \sigma(\mathcal{L}_{j,0})$ if and only if there exists a nonzero solution-pair $(P_0, Q_0) \in (L_0^2(\mathbb{C}))^2$ of the eigenvalue problem 
\begin{align} \label{linear-equation-no-control-P0}
\mu  \, P_0 = \, & \Delta_{0} P_0 -\frac{m^2}{a^2} P_0 + \lambda  \left(1-3 u_j^2\right)  P_0,
\\ \label{linear-equation-no-control-Q0}
\mu \, Q_0  = \, & \Delta_{0} Q_n - \frac{m^2}{a^2} Q_0 +  \lambda  \left(1- u_j^2\right)  Q_0.
\end{align}
Since the system \eqref{linear-equation-no-control-P0}--\eqref{linear-equation-no-control-Q0} decouples, the principal eigenvalue $\mu_{j,0}^*$ of $\mathcal{L}_{j,0}$ is strictly smaller than the principal eigenvalue of $\Delta_0 + \lambda \, (1-u_j^2)$ on $L_0^2(\mathbb{C})$, which is equivalent to \begin{equation} \label{linear-equation-on-Lm}
\Delta_m + \lambda \, (1- u_j^2): L_m^2(\mathbb{C}) \rightarrow L_m^2(\mathbb{C})
\end{equation}
as we shift the index of the Fourier modes back by $(P_0, Q_0) \mapsto (P_0 \, e^{im \varphi}, Q_0 \, e^{im \varphi})$. Notice that the gauge symmetry \eqref{gauge-symmetry} always yields zero as a trivial eigenvalue of \eqref{linear-equation-on-Lm}. 

The operator \eqref{linear-equation-on-Lm} is a singular Sturm--Liouville operator because $a(0) = 0$ (and also $a(s_*) = 0$ if $\partial \mathcal{M}$ is empty); see \eqref{restricted-operator}. However, it is singular merely because of polar coordinates \eqref{polar-coordinates}, and one expects that it has the same spectral structure as regular Sturm--Liouville operators, as we assert in the following lemma. 

\begin{lemma}[Spectral structure on $L_m^2(\mathbb{C})$]
\label{lemma-spectral-no-control-Lm}
The following statements hold:
\begin{itemize}
\item[\emph{(i)}] All eigenvalues of the self-adjoint operator \eqref{linear-equation-on-Lm} are simple. Moreover, the unstable dimension of \eqref{linear-equation-on-Lm} is $j \in \mathbb{N}_0$. Consequently, all nontrivial eigenvalues can be ordered as follows:
\begin{equation} \label{eigenvalue-order-Lm}
... <\mu_k^m <... <\mu_{j}^m < 0 < \mu_{j-1}^m < ... < \mu_0^m, \quad \lim_{k \rightarrow \infty} \mu^m_k = -\infty.
\end{equation}
\item[\emph{(ii)}] Suppose that $\partial \mathcal{M}$ is empty and the reflection symmetry \eqref{reflection-symmetry} holds. Let $y_k(s) \, e^{i m \varphi}$ be an eigenfunction of \eqref{linear-equation-on-Lm} associated with $\mu_k^m \in \mathbb{R}$. Then
\begin{equation} \label{eigenvalue-symmetry-Lm}
y_k(s_* - s) = (-1)^k y_k(s) \quad \mbox{for   }  k \in \mathbb{N}_0, \, s \in [0, s_*].
\end{equation}
\end{itemize}
\end{lemma}
\begin{proof}[Proof.]
Our proof is based on the shooting argument in \cite{DaLa21}, which has been used to prove the same spectral structure for another operator $\Delta_m + \lambda \, (1-3 u_j^2): L_m^2(\mathbb{C}) \rightarrow L_m^2(\mathbb{C})$ that differs from \eqref{linear-equation-on-Lm} only by a constant coefficient. Indeed, with the shooting argument we can obtain a monotonicity result of shooting curves, which is analogous to \cite[Lemma 3.5]{DaLa21} and thus ensures three properties explained below. 

First, the eigenvalue problem of \eqref{linear-equation-on-Lm} possesses at most one bounded nontrivial solution in $L_m^2(\mathbb{C})$. Hence all eigenvalues are simple due to the self-adjointness of \eqref{linear-equation-on-Lm}. 

Second, the unstable dimension of \eqref{linear-equation-on-Lm} is equal to the nodal class of the eigenfunction associated with the trivial eigenvalue $\mu = 0$. Observe that $\psi_j$ solves \eqref{vortex-equation} and thus is an eigenfunction of \eqref{linear-equation-on-Lm} associated with the trivial eigenvalue $\mu = 0$. Since the radial part $u_j(s)$ of $\psi_j(s, \varphi)$ possesses $j$ simple zeros on $(0, s_*)$ (see Lemma \ref{lemma-existence} (ii)), $j$ is the unstable dimension of \eqref{linear-equation-on-Lm}. As a result, the statement in (i) is proved. 

Third, $y_k(s)$ possesses exactly $k$ simple zeros on $(0, s_*)$. Observe that the eigenvalue problem of \eqref{linear-equation-on-Lm} is unchanged as we apply the new variable $s \mapsto s_* -s$, due to the $\mathbb{Z}_2$-radial-symmetry in Lemma \ref{lemma-existence} (i). Since all eigenvalues are simple by (i), either $y_k(s_* - s) = y_k(s)$ or $y_k(s_* - s) = - y_k(s)$ for $s \in [0, s_*]$. Since $y_k(s)$ possesses exactly $k$ simple zeros on $(0, s_*)$, $k \in \mathbb{N}_0$ is even if and only if $s = s_*/2$ is not a zero of $y_k(s)$, and thus if and only if $y_k(s_* - s) = y_k(s)$. The proof is complete.
\end{proof}

\subsection{Spatio-temporal feedback stabilization} \label{subsec;stabilization}

We now consider the following variational Ginzburg--Landau equation with control:
\begin{equation} \label{control-system-variational}
\partial_t \Psi =  \Delta_{\mathcal{M}} \Psi + \lambda  \left(1-\lvert\Psi\rvert^2 \right) \Psi + b \, \left(\Psi - h \, \Psi(t-\tau, R_{\iota}(s), \varphi - \zeta) \right).
\end{equation}

\subsubsection{Nodal class: \texorpdfstring{$j = 0$}{Lg}.} \label{subsec;circular}

For this class we choose $h = e^{im \zeta}$ and $R_+(s) = s$ so that the control term in \eqref{control-system-variational} is noninvasive; see also \eqref{noninvasive-variational}. The local stability of $\psi_0$ under the dynamics of the control system \eqref{control-system-variational} is determined by solutions of the following linear partial delay differential equation (see \cite[Section 4.4, Theorem 4.1]{Wu96}):
\begin{align} \label{linear-equation-control-0-V}
\begin{split}
\partial_t V & =  \Delta_\mathcal{M} V + \lambda  \left( \left(1- 2 u_0^2\right) V - u_0^2 \, e^{2 i m \varphi} \, \overline{V} \right) 
\\&
\quad + b  \left(V - e^{im\zeta} \, V(t-\tau, s, \varphi - \zeta)\right).
\end{split}
\end{align}

We aim to show that the spectrum of the linearization operator of \eqref{control-system-variational} at $\psi_0$, i.e., the right-hand side of \eqref{linear-equation-control-0-V}, consists of eigenvalues only. We then derive the characteristic equations for those eigenvalues, where $\tau \ge 0$ and $\zeta \in S^1$ act as parameters.

We shift the index of the Fourier modes by $W(t,s,\varphi) := V(t,s,\varphi) \, e^{-im\varphi}$ and set $W = P +i Q$ where $P, Q$ are real-valued functions. Then \eqref{linear-equation-control-0-V} is equivalent to 
\begin{align} \label{linear-equation-control-0-P}
\begin{split}
\partial_t P &=  \Delta_{\mathcal{M}} P - \frac{2m}{a^2} \partial_\varphi Q - \frac{m^2}{a^2} P + \lambda  \left(1-3 u_0^2\right)  P 
\\&
\quad + b  \left(P - P(t-\tau, s, \varphi - \zeta)\right),
\end{split}
\\ \label{linear-equation-control-0-Q}
\begin{split}
\partial_t Q & = \Delta_{\mathcal{M}} Q + \frac{2m}{a^2} \partial_\varphi P - \frac{m^2}{a^2} Q + \lambda \left(1- u_0^2\right) Q
\\&
\quad + b  \left(Q - Q(t-\tau, s, \varphi - \zeta)\right).
\end{split}
\end{align} 
Due to \cite[Section 3.1, Theorem 1.6]{Wu96} and the Fourier decomposition \eqref{fourier-decomposition} we can substitute the Ansatz 
\begin{align} \label{fourier-ansatz}
P(t,s,\varphi) = e^{t (\mu+i \nu)}\sum_{n \in \mathbb{Z}} P_n(s,\varphi), \quad 
Q(t,s,\varphi) =  e^{t (\mu + i \nu)}\sum_{n \in \mathbb{Z}} Q_n(s,\varphi),
\end{align}
into \eqref{linear-equation-control-0-P}--\eqref{linear-equation-control-0-Q} for $P_n, Q_n \in L_n^2(\mathbb{C})$, which yields countably many eigenvalue problems on $(L_n^2(\mathbb{C}))^2$ indexed by $n \in \mathbb{Z}$, with the eigenvalue $\mu + i\nu \in \mathbb{C}$ for $\mu, \nu \in \mathbb{R}$: 
\begin{align} \label{linear-equation-control-0-Pn}
\begin{split}
(\mu + i \nu) \, P_n & = \Delta_{n} P_n - \frac{2imn}{a^2} Q_n - \frac{m^2}{a^2} P_n + \lambda  \left(1-3 u_0^2\right)  P_n 
\\&
\quad + b\, \left(1 - e^{-\tau \mu - i (\tau \nu + n \zeta)}\right) P_n,
\end{split}
\\ \label{linear-equation-control-0-Qn}
\begin{split}
(\mu + i \nu) \, Q_n  & =  \Delta_{n} Q_n + \frac{2im n}{a^2} P_n - \frac{m^2}{a^2} Q_n + \lambda  \left(1- u_0^2\right)  Q_n \\&
\quad + b\, \left(1 - e^{-\tau \mu - i (\tau \nu + n \zeta)}\right) Q_n.
\end{split}
\end{align} 
Note that, equivalently, $\mu + i \nu \in \mathbb{C}$ is an eigenvalue of the infinitesimal generator associated with the partial delay differential equations \eqref{linear-equation-control-0-P}--\eqref{linear-equation-control-0-Q}; see \cite[Chapter 3]{Wu96}. Hence $\psi_0$ is stabilized, i.e., it becomes locally exponentially stable under the dynamics of the control system \eqref{control-system-variational}, if all nontrivial eigenvalues $\mu + i\nu \in \mathbb{C}$ in \eqref{linear-equation-control-0-Pn}--\eqref{linear-equation-control-0-Qn} satisfy $\mu < 0$ for each $n \in \mathbb{Z}$ and also the trivial eigenvalue $\mu + i \nu = 0$ triggered by the gauge symmetry \eqref{gauge-symmetry} is algebraically simple.

\begin{lemma}[Characteristic equations] \label{lemma-characteristic-0}
Let $\mathcal{L}_{0,n}$ be the operator defined as the right-hand side of \eqref{linear-equation-control-0-Pn}--\eqref{linear-equation-control-0-Qn} with $b = 0$. Then $\mu + i \nu \in \mathbb{C}$ is an eigenvalue in \eqref{linear-equation-control-0-Pn}--\eqref{linear-equation-control-0-Qn} if and only if $\mu, \nu \in \mathbb{R}$ satisfy the following \textit{characteristic equations}:
\begin{align} \label{char-equation-0-mu}
\mu  & =  \hat{\mu} + b \, \left( 1-  e^{-\tau \mu} \cos(\tau \nu + n \zeta)\right),
\\ \label{char-equation-0-nu}
\nu   & = \, b \, e^{-\tau \mu} \sin(\tau \nu + n \zeta),
\end{align} 
for some $\hat{\mu} \in \sigma(\mathcal{L}_{0,n})$.
\end{lemma}
\begin{proof}[Proof.]
The right hand side of  \eqref{linear-equation-control-0-Pn}--\eqref{linear-equation-control-0-Qn} is the sum of the operators $\mathcal{L}_{0, n}$ and $b  \, (1 - e^{-\tau \mu - i (\tau \nu + n \zeta)}) \, \mathcal{I}_n$, where $\mathcal{I}_n :(L_n^2(\mathbb{C}))^2 \rightarrow (L_n^2(\mathbb{C}))^2$ is the identity operator. Since $\mathcal{L}_{0, n}$ and $b \, (1 - e^{-\tau \mu - i (\tau \nu + n \zeta)}) \, \mathcal{I}_n$ commute, $(P_n, Q_n)$ solves the eigenvalue problem \eqref{linear-equation-control-0-Pn}--\eqref{linear-equation-control-0-Qn} with $\mu + i \nu \in \mathbb{C}$ if and only if 
$(P_n, Q_n)$ is an eigenfunction of $\mathcal{L}_{0,n}$ associated with an eigenvalue $\hat{\mu} \in \sigma(\mathcal{L}_{0,n})$ and $\mu, \nu \in \mathbb{R}$ satisfy \eqref{char-equation-0-mu}--\eqref{char-equation-0-nu}.
\end{proof}

We next show that control with pure time delays (i.e., $\zeta  = 0$ in \eqref{control-system-variational}) never achieves stabilization. Hence space shifts play an indispensable role for stabilization. 

\begin{lemma}[Failure of stabilization by control with pure time delays] \label{lemma-failure} Let $\psi_0(s, \varphi) = u_0(s) \, e^{im\varphi}$ be an unstable $m$-armed vortex equilibrium obtained in Lemma \ref{lemma-existence}. Then $\psi_0$ is not a locally exponentially stable equilibrium of the following control system:
\begin{equation} \label{control-system-no-space}
\partial_t \Psi = \Delta_{\mathcal{M}} \Psi + \lambda  \left(1-\lvert \Psi \rvert^2 \right)  \Psi + b \left( \Psi - \Psi(t-\tau, s, \varphi) \right)
\end{equation}
for all $b \in \mathbb{R}$ and $\tau \ge 0 $. 
\end{lemma}
\begin{proof}[Proof.]
Since $\psi_0$ is an unstable solution of \eqref{control-system-no-space} with $b = 0$, its associated principal eigenvalue $\mu_0^*$ is nonnegative; see Lemma \ref{lemma-spectral-no-control}.
If $\tilde{\mu} \in \mathbb{R}$ is a zero of the function
\begin{equation}
    J(\mu) := \mu - \mu_0^* + b \left(1 -e^{-\tau \mu}\right),
\end{equation} 
then \eqref{char-equation-0-mu}--\eqref{char-equation-0-nu} is satisfied with $\mu = \tilde{\mu}, \nu = 0, \hat{\mu} = \mu_0^*$, and $\zeta = 0$. So in other words, if $\tilde{\mu} \in \mathbb{R}$ is a zero of $J$, then $\tilde{\mu}$ is an eigenvalue in 
\eqref{linear-equation-control-0-Pn}--\eqref{linear-equation-control-0-Qn}. 

If $\mu_0^\ast > 0$, then $J(0) = - \mu_0^* < 0$. Since $\lim_{\mu \rightarrow \infty} J(\mu) = \infty$, the continuity of $J$ yields a $\tilde{\mu} > 0$ such that $J(\tilde{\mu}) = 0$. Hence $\tilde{\mu} > 0$ is an eigenvalue in \eqref{linear-equation-control-0-Pn}--\eqref{linear-equation-control-0-Qn}. If $\mu_0^\ast =0$, then also $J(0) = 0$ and $\tilde{\mu} = 0$ is an eigenvalue in \eqref{linear-equation-control-0-Pn}--\eqref{linear-equation-control-0-Qn}. So in both cases, \eqref{linear-equation-control-0-Pn}--\eqref{linear-equation-control-0-Qn} has an eigenvalue $\tilde{\mu} \geq 0$ and hence $\psi_0$ is not a locally exponentially stable solution of the control system \eqref{control-system-no-space}.
\end{proof}

\begin{lemma}[Selective stabilization by pure space shifts for $j = 0$] \label{lemma-pure-spatial-stabilization-0} 
Fix $m \in \mathbb{N}$, $\lambda > \lambda_0^m$, and let 
$\psi_0(s, \varphi) = u_0(s) \, e^{im\varphi}$ be the $m$-armed vortex equilibrium obtained in Lemma \ref{lemma-existence}. Then for all but finitely many choices of $\zeta \in S^1$, there exists a constant $\tilde{b} = \tilde{b}(\zeta) < 0$ such that $\psi_0$ becomes a locally exponentially stable equilibrium of the control system
\begin{equation} \label{control-system-variational-0-no-tau}
\begin{split}
\partial_t \Psi & = \Delta_{\mathcal{M}} \Psi + \lambda  \left(1-\lvert\Psi\rvert^2 \right) \Psi + b \, \left(\Psi - e^{i m \zeta} \, \Psi(t, s, \varphi - \zeta) \right)
\end{split}
\end{equation}
for all $b \le \tilde{b}$.
\end{lemma}
\begin{proof}[Proof.]
Consider the \textit{nonresonant cases} $n \neq 0$. In the equation \eqref{char-equation-0-mu}, by the inequality \eqref{largest-constant} we have 
\begin{equation} \label{inequality-0-case-1}
\mu \le \mu_0^* + b \left( 1 - \cos( n \zeta) \right).
\end{equation}
Since we consider $b \le 0$, it holds that $b\left( 1 - \cos( n \zeta) \right) \le 0$, and thus the control term does not introduce any additional instability. Therefore, by Lemma \ref{lemma-spectral-no-control} we only need to stabilize the unstable and center eigenspaces of $\mathcal{L}_{0, n}$ for those $n \neq 0$ with $-n_0\le n \le n_0$. For each such $n \in \mathbb{Z}$ the relation $1 - \cos(n\zeta) > 0$, or equivalently, 
\begin{equation} 
 n \zeta \not\equiv 0 \quad \mbox{(mod   $2\pi$)} 
\end{equation}
has all but finitely many solutions $\zeta \in S^1$. Hence $1 - \cos(n\zeta) > 0$ for $n \neq 0$ with $-n_0 \le n \le n_0$ holds for all but finitely many $\zeta \in S^1$. As we fix one such $\zeta \in S^1$, since $\mu_0^*$ is fixed, there exists a $\tilde{b} = \tilde{b}(\zeta) < 0$ such that $\mu < 0$ in \eqref{inequality-0-case-1} holds for $b \le \tilde{b}$ and $n \in \mathbb{Z} \setminus \{0\}$.

In the \textit{resonant case} $n = 0$ the control term vanishes. It suffices to consider the operator \eqref{linear-equation-on-Lm} with $j = 0$, and Lemma \ref{lemma-spectral-no-control-Lm} (i) implies that $\psi_0$ is already locally exponentially stable in $L_m^2(\mathbb{C})$. The proof is complete. 
\end{proof}

\begin{lemma}[Persistence of stabilization under small time delays for $j = 0$] \label{lemma-time-delay-perturbation-0} Consider the same setting and choices of $\zeta \in S^1$ and $\tilde{b} = \tilde{b}(\zeta) < 0$ as in Lemma \ref{lemma-pure-spatial-stabilization-0}. Then each $b \le \tilde{b}$ admits a constant $\tilde{\tau} = \tilde{\tau}(\zeta, b) > 0$ for which $\psi_0$ becomes a locally exponentially stable equilibrium of the control system
\begin{equation} \label{control-system-variational-0}
\begin{split}
\partial_t \Psi & = \Delta_{\mathcal{M}} \Psi + \lambda \left(1-\lvert\Psi\rvert^2 \right) \Psi + b \, \left(\Psi - e^{i m \zeta} \, \Psi(t-\tau, s, \varphi - \zeta) \right)
\end{split}
\end{equation}
for all $\tau \in [0, \tilde{\tau})$. 
\end{lemma}
\begin{proof}[Proof.]
As a preparation for the proof, when $\tau = 0$, for each choice of $\zeta \in S^1 $ and $b \leq \tilde{b} < 0$ that achieves stabilization in Lemma \ref{lemma-pure-spatial-stabilization-0}, there exists a $\delta > 0$ such that every nontrivial eigenvalue $\mu + i \nu \in \mathbb{C}$ in the characteristic equations \eqref{char-equation-0-mu}--\eqref{char-equation-0-nu} satisfies $\mu < -\delta < 0$. The trivial eigenvalue $\mu + i \nu = 0$ in \eqref{char-equation-0-mu}--\eqref{char-equation-0-nu} is associated with the eigenfunction belonging to $L_m^2(\mathbb{C})$, and so it is algebraically simple by Lemma \ref{lemma-spectral-no-control-Lm} (i). 

Now consider the case $\tau > 0$ and fix a choice of $\zeta \in S^1 $ and $b \leq \tilde{b} < 0$ as in Lemma \ref{lemma-pure-spatial-stabilization-0}. We prove that there exists a constant $\tilde{\tau} = \tilde{\tau}(\zeta, b) > 0$ such that all nontrivial solutions $\mu + i \nu \in \mathbb{C}$ of \eqref{char-equation-0-mu}--\eqref{char-equation-0-nu} with $\tau \in [0, \tilde{\tau})$ lie in the left-half plane $\{z \in \mathbb{C}: \mathrm{Re}(z) < -\delta\}$. 

To that end, we first prove that there exists a $\underline{\tau} \geq 0$ such that if $\hat{\mu} \in \sigma(\mathcal{L}_{0, n})$ satisfies $\hat{\mu} \leq -2 \delta$ and $\tau \in [0, \underline{\tau})$, then any solution $\mu + i \nu \in \mathbb{C}$ of \eqref{char-equation-0-mu}--\eqref{char-equation-0-nu} satisfies $\mu < - \delta$. Indeed, suppose by contradiction that there would exist a $\hat{\mu} \in \sigma(\mathcal{L}_{0, n})$ with $\hat{\mu} \leq -2 \delta$ and sequences $(\mu_\ell + i \nu_\ell)_{\ell \in \mathbb{N}}$ and $(\tau_\ell)_{\ell \in \mathbb{N}}$ with the following three properties:
\begin{itemize}
    \item $\mu_\ell + i \nu_\ell \in \mathbb{C}$ is a solution of \eqref{char-equation-0-mu}--\eqref{char-equation-0-nu} with $\tau = \tau_\ell > 0$;
    \item $\lim_{\ell \to \infty} \tau_\ell = 0$;
    \item $\mu_\ell \geq - \delta$ for all $\ell \in \mathbb{N}$.
\end{itemize}
Then squaring the characteristic equations \eqref{char-equation-0-mu}--\eqref{char-equation-0-nu} yields 
\begin{equation} \label{square-relation}
(\lvert b\rvert + \mu_\ell - \hat{\mu})^2 + \nu_\ell^2 = b^2 \, e^{-2 \tau_\ell \mu_\ell} \le b^2 \, e^{2 \tau_\ell \delta}.
\end{equation}
Since $b  = - \lvert b \rvert$ and we have assumed $\hat{\mu} \le -2\delta$ and $\mu_\ell \ge -\delta$, it holds that
\begin{align} \label{lower-bound}
\tau_\ell \ge \frac{1}{2\delta} \log \left( \left( 1 + \frac{\mu_\ell}{\lvert b\rvert} - \frac{\hat{\mu}}{\lvert b\rvert} \right)^2 + \frac{\nu_\ell^2}{b^2} \right)
\ge \frac{1}{\delta}\log \left( 1 + \frac{\delta}{\lvert b\rvert}\right)
> 0.
\end{align}
But \eqref{lower-bound} contradicts $\lim_{\ell \rightarrow \infty} \tau_\ell = 0$, since the positive lower bound 
\begin{equation}
\underline{\tau} := \frac{1}{\delta}\log \left( 1 + \frac{\delta}{\lvert b\rvert}\right)
\end{equation} 
of $\tau_\ell$ is independent of $\ell \in \mathbb{N}$ and $\hat{\mu} \le -2\delta$. We conclude that if $\hat{\mu} \in \sigma(\mathcal{L}_{0, n})$ satisfies $\hat{\mu} \le - 2 \delta$ and $\tau \in [0, \underline{\tau})$, then all nontrivial eigenvalues $\mu + i \nu \in \mathbb{C}$ in \eqref{char-equation-0-mu}--\eqref{char-equation-0-nu} satisfy $\mu < -\delta$.

Since the operator $\mathcal{L}_0$ in \eqref{linear-equation-no-control} is sectorial, only finitely many eigenvalues 
\begin{equation}
\{0\} \cup \{\hat{\mu}_q \neq 0: q = 1,2,...,\tilde{q}\} 
\end{equation}
of $\mathcal{L}_0$ lie in the right-half plane $\{z \in \mathbb{C}: \mathrm{Re}(z) > -2\delta\}$. Since the trivial eigenvalue $\mu + i \nu = 0$ in \eqref{char-equation-0-mu}--\eqref{char-equation-0-nu} with $\tau = 0$ is algebraically simple, there exists a $\tau_0 >0$ such that $\mu + i \nu = 0$ in \eqref{char-equation-0-mu}--\eqref{char-equation-0-nu} is still algebraically simple for $\tau \in [0, \tau_0)$. On the other hand, the finitely many nontrivial eigenvalues $\mu + i \nu \in \mathbb{C}$ in \eqref{char-equation-0-mu}--\eqref{char-equation-0-nu} with $\hat{\mu} = \hat{\mu}_q$ and $\tau = 0$ satisfy $\mu < -\delta$, due to stabilization by pure space shifts in Lemma \ref{lemma-pure-spatial-stabilization-0}. 
Since the eigenspace associated with these finitely many eigenvalues is finite-dimensional, each eigenvalue $\mu + i \nu \in \mathbb{C}$ in \eqref{char-equation-0-mu}--\eqref{char-equation-0-nu} with $\hat{\mu} = \hat{\mu}_q$ depends upper-semicontinuously on $\tau \ge 0$; see \cite[Chapter 3, Remark 3.3]{Ka95} or \cite[Theorem 4.4]{Sm11}. As a result, there exists a $\tau_q > 0$ such that all nontrivial eigenvalues $\mu + i \nu \in \mathbb{C}$ in \eqref{char-equation-0-mu}--\eqref{char-equation-0-nu} with $\hat{\mu} = \hat{\mu}_q$ and $\tau \in [0, \tau_q)$ satisfy $\mu < -\delta$. We complete the proof by defining $\tilde{\tau} := \min\{\underline{\tau}, \tau_0, \tau_1, ...,\tau_{\tilde{q}}\} > 0$.
\end{proof}

\begin{proof}[\textbf{Proof of Theorem \ref{theorem-0}}.] 
It remains to prove that the spatio-temporal stabilization in Lemma \ref{lemma-time-delay-perturbation-0} persists under sufficiently small parameters $0 \le \lvert\eta\rvert,\,\lvert\beta\rvert \ll 1$ in the control system \eqref{control-system-0}, as we keep the choices $\zeta \in S^1$, $b \le \tilde{b} < 0$, and $\tau \in [0, \tilde{\tau})$ as in Lemma \ref{lemma-time-delay-perturbation-0}. Such a persistence result on parameters $(\eta, \beta)$ is similar to Lemma \ref{lemma-upper-semicontinuous}, but here we prove it for the control system with a time delay $\tau > 0$.

The local stability of the selected spiral wave $\Psi_0(t,s,\varphi \,\lvert\, \eta, \beta) = e^{-i \Omega(\eta, \beta)t} \, \psi_0(s,\varphi \,\lvert\,\eta,\beta)$ under the dynamics of \eqref{control-system-0} is determined by solutions of the following linear partial delay differential equation (see \cite[Section 4.4, Theorem 4.1]{Wu96}):
\begin{equation} \label{linearization-equation-general-0}
\partial_t V = \mathcal{L}_0(\eta, \beta)[V] + b  \left(V - e^{i(-\Omega(\eta, \beta) \tau + m\zeta)} \, V(t-\tau, s, \varphi - \zeta)\right),
\end{equation}
where $\mathcal{L}_0(\eta, \beta)$ denotes the linearization operator without control; see \eqref{linearization-general}. Tuning $(\eta, \beta) \in \mathbb{R}^2$ away from $(0,0)$ yields two kinds of additional terms in \eqref{linearization-equation-general-0}: $\mathcal{L}_0(\eta, \beta) - \mathcal{L}_0(0,0)$ and the multiplicative constant $e^{-i \Omega(\eta, \beta) \tau}$; compare \eqref{linearization-equation-general-0} with \eqref{linear-equation-control-0-V}. Since $\tau \in [0, \tilde{\tau})$ is a fixed discrete time delay and thus the additional terms do not affect the functional setting of \eqref{linearization-equation-general-0}, it follows that \eqref{linearization-equation-general-0} generates a linear semiflow $\{\mathcal{S}(t \,\lvert\, \eta, \beta)\}_{t \ge 0}$ on $C^0([-\tau, 0], L^2(\mathcal{M}, \mathbb{C}))$, which becomes compact for each fixed $t > \tau$; see \cite[Section 2.1, Theorem 1.8]{Wu96}. 

It suffices to show that the spectrum of $\mathcal{S}(t \,\lvert\, \eta, \beta)$ depends upper-semicontinuously on the parameters $(\eta, \beta)$ for each fixed $t > \tau$. Since the additional terms yield perturbations only on the coefficients of \eqref{linearization-equation-general-0}, $\mathcal{S}(t \,\lvert\, \eta, \beta)$ converges to $\mathcal{S}(t \,\lvert\, 0,0)$ in the operator norm for each fixed $t > \tau$; see the argument in the proof of Lemma \ref{lemma-upper-semicontinuous}. Hence the spectrum of $\mathcal{S}(t \,\lvert\, \eta, \beta)$ depends upper-semicontinuously on the parameters $(\eta, \beta)$ for each fixed $t > \tau$; see \cite[Chapter 4, Remark 3.3]{Ka95}. The proof is complete.
\end{proof}

\subsubsection{Nodal class: \texorpdfstring{$j=1$}{Lg} and \texorpdfstring{$\partial \mathcal{M}$}{Lg} is empty.} \label{subsec;spherical}
 
In this class we choose $h = - e^{im \zeta}$ with $R_-(s) = s_*- s$ such that the control term in \eqref{control-system-variational} is noninvasive; see also \eqref{noninvasive-variational}. The local stability of $\psi_j$ under the dynamics of \eqref{control-system-variational} is determined by solutions of 
\begin{align} \label{linear-equation-control-1-V}
\begin{split}
\partial_t V &=  \Delta_\mathcal{M} V + \lambda  \left( \left(1- 2 \,u_1^2 \right) V - u_1^2 \, e^{2 i m \varphi} \, \overline{V} \right) 
\\&
\quad + b  \left(V +  e^{im\zeta} \, V(t-\tau, s_* - s, \varphi - \zeta)\right).
\end{split}
\end{align}

We again shift the index of the Fourier modes by $W(t,s,\varphi) := V(t,s,\varphi) \, e^{-im\varphi}$ and write $W = P +i Q$ where $P, Q$ are real-valued functions. Then \eqref{linear-equation-control-1-V} is equivalent to 
\begin{align} \label{linear-equation-control-1-P}
\begin{split}
\partial_t P &=  \Delta_{\mathcal{M}} P - \frac{2m}{a^2} \partial_\varphi Q - \frac{m^2}{a^2} P + \lambda  \left(1-3 u_1^2\right)  P 
\\&
\quad + b  \left(P + P(t-\tau, s_* -s, \varphi - \zeta)\right),
\end{split}
\\ \label{linear-equation-control-1-Q}
\begin{split}
\partial_t Q & = \Delta_{\mathcal{M}} Q + \frac{2m}{a^2} \partial_\varphi P - \frac{m^2}{a^2} Q + \lambda \left(1- u_1^2\right)  Q 
\\&
\quad + b  \left(Q + Q(t-\tau, s_* - s, \varphi - \zeta)\right).
\end{split}
\end{align} 

By \cite[Section 3.1, Theorem 1.6]{Wu96} and the Fourier decomposition \eqref{fourier-decomposition} substituting the Ansatz \eqref{fourier-ansatz} into \eqref{linear-equation-control-1-P}--\eqref{linear-equation-control-1-Q} yields countably many eigenvalue problems on $(L_n^2(\mathbb{C}))^2$ indexed by $n \in \mathbb{Z}$, for the eigenvalue $\mu + i\nu \in \mathbb{C}$: 
\begin{align} \label{linear-equation-control-1-Pn}
\begin{split}
(\mu + i \nu) \, P_n & =  \Delta_{n} P_n - \frac{2imn}{a^2} Q_n - \frac{m^2}{a^2} P_n + \lambda  \left(1-3 u_1^2\right)  P_n 
\\&
\quad + b\, \left(1 + e^{-\tau \mu - i (\tau \nu + n \zeta)} \mathcal{R}_n\right) P_n,
\end{split}
\\ \label{linear-equation-control-1-Qn}
\begin{split}
(\mu + i \nu) \, Q_n  & =  \Delta_{n} Q_n + \frac{2im n}{a^2} P_n - \frac{m^2}{a^2} Q_n + \lambda  \left(1- u_1^2\right)  Q_n 
\\&
\quad + b\, \left(1 + e^{-\tau \mu - i (\tau \nu + n \zeta)} \mathcal{R}_n\right) Q_n.
\end{split}
\end{align} 
Here $\mathcal{R}_n : L_n^2(\mathbb{C}) \rightarrow L_n^2(\mathbb{C})$ is the \textit{reflection operator} $L_n^2(\mathbb{C})$ defined by 
\begin{equation} 
(\mathcal{R}_n[P_n])(s, \varphi) := P_n(s_*-s, \varphi).
\end{equation}

\begin{lemma}[Characteristic equations] \label{lemma-characteristic-1}
Let $\mathcal{L}_{1,n}$ be the operator defined as the right-hand side of \eqref{linear-equation-control-1-Pn}--\eqref{linear-equation-control-1-Qn} with $b = 0$. Then $\mu + i \nu \in \mathbb{C}$ is an eigenvalue in \eqref{linear-equation-control-1-Pn}--\eqref{linear-equation-control-1-Qn} if $\mu, \nu \in \mathbb{R}$ satisfy the following \textit{characteristic equations}:
\begin{align} \label{char-equation-1-mu}
\mu  & =  \hat{\mu} + b \, \left( 1 + \chi \,  e^{-\tau \mu} \cos(\tau \nu + n \zeta)\right),
\\ \label{char-equation-1-nu}
\nu   & = \chi \, b \, e^{-\tau \mu} \sin(\tau \nu + n \zeta),
\end{align} 
for some $\hat{\mu} \in \sigma(\mathcal{L}_{1,n})$ and some $\chi \in \{-1, 1\}$.
\end{lemma}
\begin{proof}[Proof.]
Observe that $(P_n, Q_n)$ is an eigenfunction in \eqref{linear-equation-control-1-Pn}--\eqref{linear-equation-control-1-Qn} if and only if $(\mathcal{R}_n[P_n], \mathcal{R}_n[Q_n])$ is also an eigenfunction, due to the relation $u_1^2(s_* - s) = u_1^2(s)$ in Lemma \ref{lemma-existence} (i). Define 
\begin{align} \label{eigenfunction-even}
(P^e_n, Q^e_n) &= (P_n + \mathcal{R}_n[P_n], Q_n + \mathcal{R}_n[Q_n]),
\\ \label{eigenfunction-odd}
(P^o_n, Q^o_n) &= (P_n - \mathcal{R}_n[P_n], Q_n - \mathcal{R}_n[Q_n]).
\end{align} 
Then either $(P^e_n, Q^e_n)$ or $(P^o_n, Q^o_n)$ is a nonzero solution-pair and thus is an eigenfunction in \eqref{linear-equation-control-1-Pn}--\eqref{linear-equation-control-1-Qn}. Therefore, $\mu + i \nu \in \mathbb{C}$ is an eigenvalue in \eqref{linear-equation-control-1-Pn}--\eqref{linear-equation-control-1-Qn} if $\mu, \nu \in \mathbb{R}$ satisfy \eqref{char-equation-1-mu}--\eqref{char-equation-1-nu}
for some $\hat{\mu} \in \sigma(\mathcal{L}_{1,n})$ and $\chi = 1$ (resp., $\chi = -1$) when $(P^e_n, Q^e_n)$ (resp., $(P^o_n, Q^o_n)$) is an eigenfunction in \eqref{linear-equation-control-1-Pn}--\eqref{linear-equation-control-1-Qn}. 
\end{proof}

We emphasize that our subsequent stabilization analysis does not rely on knowledge of the exact value of $\chi \in \{-1,1\}$ in Lemma \ref{lemma-characteristic-1}. 

\begin{lemma}[Selective stabilization by pure space shifts for $j = 1$] \label{lemma-pure-spatial-stabilization-1} 
Fix $m \in \mathbb{N}$, $\lambda > \lambda_1^m$, and let 
$\psi_1(s, \varphi) = u_1(s) \, e^{im\varphi}$ be the $m$-armed vortex equilibrium obtained in Lemma \ref{lemma-existence}. Then for all but finitely many choices of $\zeta \in S^1$, there exists a constant $\tilde{b} = \tilde{b}(\zeta) < 0$ such that $\psi_1$ becomes a locally exponentially stable equilibrium of the control system
\begin{equation} \label{control-system-variational-1-no-tau}
\begin{split}
\partial_t \Psi & = \Delta_{\mathcal{M}} \Psi + \lambda \left(1-\lvert\Psi\rvert^2 \right) \Psi + b \, \left(\Psi + e^{i m \zeta} \, \Psi(t, s_*-s, \varphi - \zeta) \right)
\end{split}
\end{equation}
for all $b \le \tilde{b}$.
\end{lemma}
\begin{proof}[Proof.]
The proof is similar to the one in Lemma \ref{lemma-pure-spatial-stabilization-0}, but it requires a careful treatment to determine the value of $\chi\in\{-1,1\}$ in the resonant case $n = 0$.

Consider the \textit{nonresonant cases} $n \neq 0$. Then the equation \eqref{char-equation-1-mu} together with the inequality \eqref{largest-constant} implies 
\begin{equation} \label{inequality-1-case-1}
\mu \le \mu_1^* + b \left( 1 + \chi \cos( n \zeta) \right).
\end{equation}
Since $\chi \in \{-1, 1\}$ and thus $b\left( 1 + \chi \cos( n \zeta) \right) \le 0$ as we consider $b \le 0$, by Lemma \ref{lemma-spectral-no-control} it suffices to stabilize the unstable and center eigenspaces of $\mathcal{L}_{1, n}$ for $n \neq 0$ with $-n_1 \le n \le n_1$. Since solutions satisfying the relations $1 - \chi \cos(n\zeta) > 0$ for $n \in \mathbb{Z}$ and $\chi \in \{-1,1\}$ form a subset of solutions of 
\begin{equation}
n \zeta \not\equiv 0 \quad \mbox{(mod   $\pi$)},     
\end{equation}
we see that $1 - \chi \cos(n\zeta) > 0$ for $n \neq 0$ with $-n_1 \le n \le n_1$ holds for all but finitely many $\zeta \in S^1$. As we fix one such $\zeta \in S^1$, since $\mu_1^*$ is fixed, there exists a $\tilde{b} = \tilde{b}(\zeta) < 0$ such that $\mu < 0$ in \eqref{inequality-1-case-1} holds for $b \le \tilde{b}$, $\chi \in \{-1, 1\}$, and $n \in \mathbb{Z} \setminus \{0\}$. 

In the \textit{resonant case} $n = 0$, the system \eqref{linear-equation-control-1-Pn}--\eqref{linear-equation-control-1-Qn} decouples, and by comparison of eigenvalues it suffices to show that all eigenvalues of
the self-adjoint operator 
\begin{equation} \label{linear-equation-on-Lm-with-R}
\Delta_m + \lambda  \left(1- u_1^2\right) + b \left(1 + \mathcal{R}_n\right) : L_m^2(\mathbb{C}) \rightarrow L_m^2(\mathbb{C})
\end{equation}
are negative for some $b < 0$.

Let $y_k(s) \, e^{im\varphi}$ be an eigenfunction of $\Delta_m + \lambda \, (1- u_1^2)$ associated with $\mu^m_k \in \mathbb{R}$; see Lemma \ref{lemma-spectral-no-control-Lm} (ii). Then the symmetry \eqref{eigenvalue-symmetry-Lm} implies that $y_k(s) \, e^{im\varphi}$ is also an eigenfunction of the operator \eqref{linear-equation-on-Lm-with-R}. Since $\Delta_m + \lambda \, (1-u_1^2)$ has compact resolvent and thus its eigenfunctions form a basis of $L_m^2(\mathbb{C})$, the operator \eqref{linear-equation-on-Lm-with-R} and $\Delta_m + \lambda \, (1-u_1^2)$ indeed share the same set of eigenfunctions, which implies that the spectrum of the operator \eqref{linear-equation-on-Lm-with-R} consists of eigenvalues, only.

Let $\mu_k \in \mathbb{R}$ be the eigenvalue of the operator \eqref{linear-equation-on-Lm-with-R} associated with the eigenfunction $y_k(s) \, e^{i m \varphi}$. From \eqref{eigenvalue-symmetry-Lm} we know 
\begin{equation} \label{eigenvalue-relation-1}
\mu_k = \mu_k^m + b \,\left(1 + (-1)^{k} \right). 
\end{equation}
Since $b < 0$, by \eqref{eigenvalue-order-Lm} and \eqref{eigenvalue-relation-1} we know $\mu_k < 0$ for $k \ge 1$. The other case $k = 0$ in \eqref{eigenvalue-relation-1} yields $\mu_0 = \mu^m_0 + 2 b$, and so $\mu_0 < 0$ for $b < -\mu_0^m/2$.  
\end{proof}

\begin{lemma}[Persistence of stabilization under small time delays for $j = 1$] \label{lemma-time-delay-perturbation-1} Consider the same setting and choices of $\zeta \in S^1$ and $\tilde{b} = \tilde{b}(\zeta) < 0$ in Lemma \ref{lemma-pure-spatial-stabilization-1}. Then each $b \le \tilde{b}$ admits a constant $\tilde{\tau} = \tilde{\tau}(\zeta, b) > 0$ for which $\psi_1$ becomes a locally exponentially stable equilibrium of the control system
\begin{equation} \label{control-system-variational-1}
\begin{split}
\partial_t \Psi & = \Delta_{\mathcal{M}} \Psi + \lambda \left(1-\lvert\Psi\rvert^2 \right) \Psi + b \, \left(\Psi + e^{i m \zeta} \, \Psi(t-\tau, s_*-s, \varphi - \zeta) \right)
\end{split}
\end{equation}
for all $\tau \in [0, \tilde{\tau})$. 
\end{lemma}
\begin{proof}[Proof.]
The proof is the same as the one in Lemma \ref{lemma-time-delay-perturbation-0}, since we obtain the same equation \eqref{square-relation} after squaring the characteristic equations \eqref{char-equation-1-mu}--\eqref{char-equation-1-nu}, no matter whether $\chi$ is $-1$ or $1$. 
\end{proof}

\begin{proof}[\textbf{Proof of Theorem \ref{theorem-1}}.] It remains to prove that the spatio-temporal stabilization in Lemma \ref{lemma-time-delay-perturbation-1} persists under sufficiently small parameters $0 \le \lvert\eta\rvert,\,\lvert\beta\rvert \ll 1$ in the control system \eqref{control-system-1}, as we keep the choices $\zeta \in S^1$, $b \le \tilde{b} < 0$, and $\tau \in [0, \tilde{\tau})$ in Lemma \ref{lemma-time-delay-perturbation-1}. Indeed, since the $\mathbb{Z}_2$-radial-symmetry in Lemma \ref{lemma-existence} (i) holds for $0 \le \lvert\eta\rvert,\,\lvert\beta\rvert \ll 1$, the proof is analogous to the one of Theorem \ref{theorem-0} with only one mild adaptation: The two multiplicative factors $h$ differ by $-1$; see \eqref{noninvasive-general}. The proof is complete.
\end{proof}

\bmhead{Acknowledgments}
I. S. and B. d.W. have been partially supported by the Deutsche Forschungsgemeinschaft, SFB 910, Project A4 “Spatio-Temporal Patterns: Control, Delays, and Design”; B. d.W. was supported by the Berlin Mathematical School. J.-Y. D. has been supported by MOST grant number 110-2115-M-005-008-MY3. We are grateful to Bernold Fiedler, Alejandro L\'{o}pez Nieto, and Jan Totz for many inspiring and fruitful discussions.

\bmhead{Data availability} Data sharing not applicable to this article as no datasets were generated or analysed during the current study.

\bmhead{Competing interests} The authors have no financial or proprietary interests in any material discussed in this article.




\end{document}